\documentclass[10pt,english,reqno]{amsart} 

\calclayout

\usepackage[usenames,dvipsnames]{color}

\usepackage[hidelinks]{hyperref}
\usepackage{graphicx}
\usepackage{amssymb}
\usepackage{epstopdf}
\usepackage{enumerate}
\usepackage{mathabx}
\usepackage{tikz-cd}
\usepackage{MnSymbol}
\usepackage[mathscr]{eucal}
\usepackage[multiple]{footmisc}

\DeclareMathSymbol{A}{\mathalpha}{operators}{`A}
\DeclareMathSymbol{B}{\mathalpha}{operators}{`B}
\DeclareMathSymbol{C}{\mathalpha}{operators}{`C}
\DeclareMathSymbol{D}{\mathalpha}{operators}{`D}
\DeclareMathSymbol{E}{\mathalpha}{operators}{`E}
\DeclareMathSymbol{F}{\mathalpha}{operators}{`F}
\DeclareMathSymbol{G}{\mathalpha}{operators}{`G}
\DeclareMathSymbol{H}{\mathalpha}{operators}{`H}
\DeclareMathSymbol{I}{\mathalpha}{operators}{`I}
\DeclareMathSymbol{J}{\mathalpha}{operators}{`J}
\DeclareMathSymbol{K}{\mathalpha}{operators}{`K}
\DeclareMathSymbol{L}{\mathalpha}{operators}{`L}
\DeclareMathSymbol{M}{\mathalpha}{operators}{`M}
\DeclareMathSymbol{N}{\mathalpha}{operators}{`N}
\DeclareMathSymbol{O}{\mathalpha}{operators}{`O}
\DeclareMathSymbol{P}{\mathalpha}{operators}{`P}
\DeclareMathSymbol{Q}{\mathalpha}{operators}{`Q}
\DeclareMathSymbol{R}{\mathalpha}{operators}{`R}
\DeclareMathSymbol{S}{\mathalpha}{operators}{`S}
\DeclareMathSymbol{T}{\mathalpha}{operators}{`T}
\DeclareMathSymbol{U}{\mathalpha}{operators}{`U}
\DeclareMathSymbol{V}{\mathalpha}{operators}{`V}
\DeclareMathSymbol{W}{\mathalpha}{operators}{`W}
\DeclareMathSymbol{X}{\mathalpha}{operators}{`X}
\DeclareMathSymbol{Y}{\mathalpha}{operators}{`Y}
\DeclareMathSymbol{Z}{\mathalpha}{operators}{`Z}

\newcommand{\integers}{\mathbf Z}

\newcommand{\reals}{\mathbf R}
\newcommand{\complexes}{\mathbf C}

\newcommand{\Cat}{\textnormal{\textsf{Cat}}}
\newcommand{\Spc}{\textnormal{\textsf{Spc}}}
\newcommand{\Set}{\textnormal{\textsf{Set}}}
\newcommand{\Sptr}{\textnormal{\textsf{Sptr}}}
\newcommand{\Alg}{\textnormal{\textsf{Alg}}}
\newcommand{\CAlg}{\textnormal{\textsf{CAlg}}}
\newcommand{\Sch}{\textnormal{\textsf{Sch}}}
\newcommand{\PStk}{\textnormal{\textsf{PStk}}}

\newcommand{\Fin}{\textnormal{\textsf{Fin}}}
\newcommand{\Rep}{\mathrm{Rep}}
\newcommand{\Top}{\textnormal{\textsf{Top}}}
\newcommand{\Shv}{\textnormal{\textsf{Shv}}}
\newcommand{\DGCat}{\textnormal{\textsf{DGCat}}}
\newcommand{\Groth}{\textnormal{\textsf{Groth}}}
\newcommand{\sep}{\mathrm{sep}}

\newcommand{\coBar}{\operatorname{coBar}}

\newcommand{\Mfd}{\textnormal{\textsf{Mfd}}}
\newcommand{\Disk}{\textnormal{\textsf{Disk}}}

\newcommand{\PrL}{\textnormal{\textsf{Pr}}^L}
\newcommand{\Vect}{\textnormal{\textsf{Vect}}}

\newcommand{\Fun}{\mathrm{Fun}}
\newcommand{\Mod}{\textnormal{-\textsf{mod}}}

\newcommand{\Emb}{\mathrm{Emb}}

\newcommand{\ev}{\mathrm{ev}}

\newcommand{\opposite}{\textnormal{op}}
\newcommand{\Maps}{\operatorname{Maps}}

\newcommand{\deloop}{\textnormal{\textsf{B}}}
\newcommand{\topology}{\mathrm{top}}

\newcommand{\ft}{\mathrm{ft}}

\DeclareMathOperator*\colim{colim}

\newcommand{\Sym}{\operatorname{Sym}}

\newcommand{\Bun}{\mathrm{Bun}}
\newcommand{\LS}{\mathrm{LS}}
\newcommand{\glob}{\mathrm{glob}}
\newcommand{\loc}{\mathrm{loc}}
\newcommand{\Sing}{\operatorname{Sing}}
\newcommand{\Lie}{\operatorname{Lie}}
\newcommand{\Level}{\mathrm{Level}}
\newcommand{\Quad}{\mathrm{Quad}}

\newcommand{\SO}{\mathrm{SO}}
\newcommand{\GL}{\mathrm{GL}}

\makeatletter
\renewcommand{\@secnumfont}{\bfseries}
\makeatother

\newtheorem{thm}[subsubsection]{Theorem}
\newtheorem*{thm*}{Theorem}

\newtheorem{prop}[subsubsection]{Proposition}
\newtheorem{lem}[subsubsection]{Lemma}

\newtheorem{cor}[subsubsection]{Corollary}

\theoremstyle{definition}

\newtheorem{rem}[subsubsection]{Remark}

\numberwithin{equation}{section}
\allowdisplaybreaks

\newtheoremstyle{void}
    {}{}{}{}
    {\bfseries}{.}{ }
    {\thmname{#1}\thmnumber{#2}\thmnote{ {\mdseries \textit{#3}}}}
\theoremstyle{void}

\newtheorem{void}[subsubsection]{}

\title{The quantum torus as an $\mathbb E_M$-category}

\author{Lin Chen \and Yifei Zhao}
\date{\today}

\begin{document}

\begin{abstract}
Given an oriented $2$-manifold $M$, a locally constant sheaf of lattices $\Lambda$ over $M$, and a pointed morphism $q : \textnormal{\textsf B}^2\Lambda \rightarrow \textnormal{\textsf B}^4\complexes^{\times}$, we define an $\mathbb E_M$-category $\mathrm{Rep}_q(\check T)$ which we call the ``quantum torus" at level $q$. We explain why this terminology is deserved and calculate the factorization homology of $\mathrm{Rep}_q(\check T)$. When $M$ arises from a global complex curve, we confirm (a version of) a conjecture of Ben-Zvi and Nadler for tori.
\end{abstract}

\maketitle

\setcounter{tocdepth}{2}
\tableofcontents


\section*{Introduction}

Given a reductive group $G$ and a ``level" $q\in\complexes^{\times}$, one obtains the category $\Rep_q(\check G)$ of representations of the quantum group. It can be regarded as a braided monoidal deformation of the category $\Rep(\check G)$ of representations of the Langlands dual group $\check G$.

Morever, the category $\Rep_q(\check G)$ admits a ribbon structure, which allows one to ``spread" $\Rep_q(\check G)$ onto an oriented $2$-manifold $M$, thus defining an $\mathbb E_M$-category. It is then possible to extract a global invariant
$$
\int_M \Rep_q(\check G),
$$
called its ``factorization homology", which plays a prominent role in the Betti quantum geometric Langlands program. We refer the reader to \cite{MR3847209, MR3821166} for details.

The consideration above makes us suspect that $\Rep_q(\check G)$ starts life naturally as an $\mathbb E_M$-category. Indeed, in the context of the quantum geometric Langlands program, the level $q$ is not expected to be a complex number in general, but a (suitably categorified) degree-$4$ reduced cohomology class of $\deloop G$. Unless $G$ is simply connected, a level of this kind can vary along $M$, so we do not expect $\Rep_q(\check G)$ to come from a single ribbon category.

In this note, we confirm this suspicion for quantum tori: We give a direct definition of $\Rep_q(\check T)$ as an $\mathbb E_M$-category, which accommodates the general notion of levels as well as nonsplit tori. From a classical perspective, one can say that this note explains how quantum tori behave in family.

\subsection*{Contents of this note}

Our definition of the ``quantum torus" $\Rep_q(\check T)$ takes as input a triple $(M, \Lambda, q)$, where
\begin{enumerate}
	\item $M$ is an oriented $2$-manifold (with underlying $\infty$-groupoid $\Sing M$);
	\item $\Lambda$ is a functor from $\Sing M$ to the category of finite free $\integers$-modules;
	\item $q : \deloop^2\Lambda \rightarrow \deloop^4\complexes^{\times}$ is a morphism in the $\infty$-category $\Fun(\Sing M, \Spc_*)$, where $\Spc_*$ denotes the $\infty$-category of pointed $\infty$-groupoids.
\end{enumerate}

Given the triple $(M, \Lambda, q)$, we shall define $\Rep_q(\check T)$ as an $\mathbb E_M$-algebra in the $\infty$-category $\DGCat$ of DG categories in \S\ref{sec-quantum-torus-definition}.

Our definition is conceptually simple, but not quite explicit. To argue that we have given a reasonable definition, we shall show that $\Rep_q(\check T)$ has the expected behavior of a quantum torus: It is completely determined by its heart $\Rep_q(\check T)^{\heartsuit}$, which is a family of ``twisted" braided monoidal categories over $M$ whose local invariants can be expressed explicitly in terms of those of $q$ (\emph{cf.}~Proposition \ref{void-ribbon-structure-context}, Proposition \ref{prop-ribbon-twist}).

As for global invariants, we shall compute the factorization homology of $\Rep_q(\check T)$. In Theorem \ref{thm-factorization-homology-quantum-torus}, we shall construct a canonical equivalence in $\DGCat$:
\begin{equation}
\label{eq-factorization-homology-intro}
\int_M \Rep_q(\check T) \simeq \LS_q(\Gamma_c(M, \deloop^2\Lambda)),
\end{equation}
where the right-hand-side denotes the DG category of ``$q$-twisted" local systems over the $\infty$-groupoid $\Gamma_c(M, \deloop^2\Lambda)$ of compactly supported sections.

In fact, our definition of $\Rep_q(\check T)$ turns \eqref{eq-factorization-homology-intro} into an immediate corollary of nonabelian Poincar\'e duality, due to Salvatore, Segal, and Lurie (\emph{cf.}~\cite{MR1851264, MR2681691, lurie2017higher}).

When $M$ is the underlying oriented $2$-manifold of a global complex curve $X$ and $\Lambda$ is defined by the (locally constant) sheaf of cocharacters of an $X$-torus $T$, the equivalence \eqref{eq-factorization-homology-intro} implies a version of \cite[Conjecture 4.27]{MR3821166}, the quantum Betti geometric Langlands conjecture for tori (\emph{cf.}~Corollary \ref{cor-ben-zvi-nadler}).

\subsection*{Acknowledgements}
We thank Dennis Gaitsgory, Sam Raskin, and Nick Rozenblyum for teaching us about the geometric Langlands program.

In addition, Y.Z.~thanks Thomas Nikolaus, Phil P\"utzst\"uck, and Maxim Ramzi for patiently answering his topology questions, and JiWoong Park for helpful conversations.

\section{Preparation}
\label{sec-preparation}

In this section, we recall the notions of $\mathbb E_M$-algebras and factorization homology, and gather all of their properties that we shall use later. These properties are comprehensively established in \cite[\S5.5]{lurie2017higher} and, from an alternative point of view, in \cite{MR3431668}. We also recall the formalism of local systems of \cite{MR4368480}.

Needless to say, this section contains no originality.

\subsection{Factorization homology}

\begin{void}
\label{void-category-of-manifolds}
Fix an integer $n\ge 0$.

Denote by $\Mfd_n$ the topological category of $n$-dimensional topological manifolds admitting finite ``good covers" (\emph{cf.}~\cite[Definition 2.1]{MR3431668}).

We shall suppress the operation of taking homotopy coherent nerves from our notation, and regard $\Mfd_n$ as an $\infty$-category.

Thus, the mapping space between $M_1, M_2 \in \Mfd_n$ is the $\infty$-groupoid $\Sing \Emb(M_1, M_2)$, where $\Emb(M_1, M_2)$ is the set of embeddings $M_1 \rightarrow M_2$, endowed with the compact-open topology, and $\Sing$ denotes the functor of singular chains.
\end{void}

\begin{void}
Denote by $\deloop\Top(n)$ the full subcategory of $\Mfd_n$ consisting of objects homeomorphic to the Euclidean space $\reals^n$.

The Kister--Mazur theorem shows that $\Top(n) := \Sing \Emb(\reals^n, \reals^n)$ is a \emph{grouplike} monoid. In particular, $\deloop \Top(n)$ may be identified with the classifying space of $\Top(n)$.

Given $M \in \Mfd_n$, we have a forgetful functor
\begin{equation}
\label{eq-tangent-classifier-construction}
\deloop\Top(n)_{/M} \rightarrow \deloop\Top(n).
\end{equation}
\end{void}

\begin{rem}
\label{rem-tangent-classifier-via-yoneda}
The slice $\infty$-category $\deloop\Top(n)_{/M}$ is a Kan complex equivalent to $\Sing M$ (\emph{cf.}~\cite[Remark 5.4.5.2]{lurie2017higher}), so \eqref{eq-tangent-classifier-construction} determines a morphism of $\infty$-groupoids
$$
\tau_M : \Sing M \rightarrow \deloop\Top(n).
$$
By \cite[Corollary 2.13]{MR3431668}, this morphism classifies the tangent microbundle of $M$.
\end{rem}

\begin{void}[The $\infty$-operad $\mathbb E_M^{\otimes}$]
Recall that $\deloop\Top(n)$ is the underlying $\infty$-category of an $\infty$-operad $\deloop\Top(n)^{\otimes}$ (\emph{cf.}~\cite[Definition 5.4.2.1]{lurie2017higher}).

Moreover, each $\infty$-category $\mathscr C$ functorially determines an $\infty$-operad $\mathscr C^{\sqcup}$ (\emph{cf.}~\cite[\S2.4.3]{lurie2017higher}) and we have a natural morphism $\deloop\Top(n)^{\otimes} \rightarrow \deloop \Top(n)^{\sqcup}$ (\emph{cf.}~\cite[Remark 5.4.2.7]{lurie2017higher}).

For each $M \in \Mfd_n$, the $\infty$-operad $\mathbb E_M^{\otimes}$ is defined as the fiber product
$$
\mathbb E_M^{\otimes} := \deloop\Top(n)^{\otimes} \times_{\deloop\Top(n)^{\sqcup}} (\deloop\Top(n)_{/M})^{\sqcup}.
$$
\end{void}

\begin{rem}
Let us give an informal description of $\mathbb E_M^{\otimes}$.

By definition, the underlying $\infty$-category of $\mathbb E_M^{\otimes}$ is $\deloop\Top(n)_{/M}$. Given objects $x : U \rightarrow M$ and $x_j : U_j \rightarrow M$ ($j = 1, \cdots, m$) of $\deloop\Top(n)_{/M}$, an $m$-ary operation from $\{x_j\}_{j = 1,\cdots, m}$ to $\{x\}$ in $\mathbb E_M^{\otimes}$ consists of an embedding $U_1\sqcup\cdots\sqcup U_m \rightarrow U$ together with an identification of the composite
$$
U_j \rightarrow U_1 \sqcup\cdots \sqcup U_m \rightarrow U \xrightarrow{x} M
$$
with $x_j$ as morphisms in $\Mfd_n$, for each $j = 1,\cdots, m$.
\end{rem}

\begin{void}
\label{void-em-algebras}
We fix $M \in \Mfd_n$ in the remainder of this section.

Let $\mathscr O$ be a symmetric monoidal $\infty$-category. We shall refer to $\mathbb E_M^{\otimes}$-algebras in $\mathscr O$ simply as \emph{$\mathbb E_M$-algebras}. They form an $\infty$-category $\Alg_{\mathbb E_M}(\mathscr O)$.

Since the $\infty$-category underlying $\mathbb E_M^{\otimes}$ is $\deloop\Top(n)_{/M}$, or equivalently $\Sing M$ (\emph{cf.}~Remark \ref{rem-tangent-classifier-via-yoneda}), we have a forgetful functor
\begin{equation}
\label{eq-em-algebra-forgetful}
\Alg_{\mathbb E_M}(\mathscr O) \rightarrow \Fun(\Sing M, \mathscr O).
\end{equation}

Given $\mathscr A \in \Alg_{\mathbb E_M}(\mathscr O)$ and $x\in\Sing M$, we write $\mathscr A_x \in \mathscr O$ for the image of $x$ under the functor underlying $\mathscr A$ and refer to it as the \emph{fiber} of $\mathscr A$ at $x$.
\end{void}

\begin{void}[Factorization homology]
When $\mathscr O$ is sifted-complete, we shall construct a functor
\begin{equation}
\label{eq-factorization-homology}
\int_M : \Alg_{\mathbb E_M}(\mathscr O) \rightarrow \mathscr O,
\end{equation}
whose value at $\mathscr A \in \Alg_{\mathbb E_M}(\mathscr O)$ is called the \emph{factorization homology} of $\mathscr A$ over $M$.

Denote by $\Disk_n$ the full subcategory of $\Mfd_n$ consisting of objects homeomorphic to $S \times \reals^n$ for some finite set $S$.

The construction of \eqref{eq-factorization-homology} relies on a functor of $\infty$-categories
\begin{equation}
\label{eq-disk-to-em-functor}
\Disk_{n/M} \rightarrow \mathbb E_M^{\otimes},
\end{equation}
which we shall define presently.\footnote{The functor \eqref{eq-disk-to-em-functor} appears implicitly in the proof of \cite[Theorem 5.5.2.5]{lurie2017higher}, but we could not locate its definition in \emph{op.cit.}.}
\end{void}

\begin{void}[Construction of \eqref{eq-disk-to-em-functor}]
It suffices to construct functors
\begin{align}
\label{eq-disk-to-em-functor-first}
	\Disk_{n/M} &\rightarrow \deloop\Top(n)^{\otimes}, \\
\label{eq-disk-to-em-functor-second}
	\Disk_{n/M} &\rightarrow (\deloop\Top(n)_{/M})^{\sqcup},
\end{align}
and identify their compositions with the functors to $\deloop\Top(n)^{\sqcup}$.

We follow the notation of \cite[\S2.1.1]{lurie2017higher} and write $\Fin_*$ for the category of pointed finite sets. We express any $S \in \Fin_*$ as $S^{\circ} \sqcup \{*\}$, where $*$ is the distinguished element. Observe that $\Disk_{n/M}$ admits a natural functor to $\Fin_*$, sending $U \rightarrow M$ to the pointed finite set $\pi_0 U \sqcup\{*\}$. The functors \eqref{eq-disk-to-em-functor-first} and \eqref{eq-disk-to-em-functor-second}, which we shall construct, intertwine this functor to $\Fin_*$ with the structural functors of the $\infty$-operads $\deloop\Top(n)^{\otimes}$ and $(\deloop\Top(n)_{/M})^{\sqcup}$.

The functor \eqref{eq-disk-to-em-functor-first} is defined as the composition of the forgetful functor $\Disk_{n/M} \rightarrow \Disk_n$ with the faithful embedding of \emph{topological} categories\footnote{The topological category defining $\deloop\Top(n)^{\otimes}$ is denoted by ${}^t\mathbb E^{\otimes}_{\deloop\Top(n)}$ in \cite[Definition 5.4.2.1]{lurie2017higher} and \eqref{eq-disk-embedding-in-btop-operad} identifies $\Disk_n$ with its faithful subcategory consisting of all objects and active morphisms.}
\begin{equation}
\label{eq-disk-embedding-in-btop-operad}
\Disk_n \rightarrow \deloop\Top(n)^{\otimes}.
\end{equation}

The functor \eqref{eq-disk-to-em-functor-second} is adjoint to a functor
\begin{equation}
\label{eq-disk-to-em-functor-second-adjoint}
\Disk_{n/M} \times_{\Fin_*} \Gamma^* \rightarrow \deloop\Top(n)_{/M},
\end{equation}
where $\Gamma^*$ denotes the category of pairs $(S, i)$, with $S \in \Fin_*$ and $i \in S^{\circ}$ (\emph{cf.}~\cite[Construction 2.4.3.1]{lurie2017higher}). The functor \eqref{eq-disk-to-em-functor-second-adjoint} sends $(U \rightarrow M, i)$ to the restriction of $U \rightarrow M$ to the connected component of $U$ corresponding to $i$.

To identify the compositions of \eqref{eq-disk-to-em-functor-first} and \eqref{eq-disk-to-em-functor-second} with the natural functors to $\deloop\Top(n)^{\sqcup}$, we observe that the composition of \eqref{eq-disk-to-em-functor-second-adjoint} with the forgetful functor to $\deloop\Top(n)$ factors as in the following commutative diagram
$$
\begin{tikzcd}[column sep = 1em]
	\Disk_{n/M} \times_{\Fin_*} \Gamma^* \ar[r]\ar[d] & \deloop\Top(n)_{/M} \ar[d] \\
	\Disk_n \times_{\Fin_*} \Gamma^* \ar[r] & \deloop\Top(n)
\end{tikzcd}
$$
Here, the bottom horizontal arrow is defined by evaluation at $i$. By construction, it is adjoint to the composition of \eqref{eq-disk-embedding-in-btop-operad} with the functor $\deloop\Top(n)^{\otimes} \rightarrow \deloop\Top(n)^{\sqcup}$.
\end{void}

\begin{void}[Construction of \eqref{eq-factorization-homology}]
\label{void-factorization-homology-construction}
The functor \eqref{eq-factorization-homology} is the composition of the restriction along \eqref{eq-disk-to-em-functor} with the functor of taking colimits over $\Disk_{n/M}$. In other words, we have
$$
\int_M \mathscr A := \colim_{U \in \Disk_{n/M}} \mathscr A(U).
$$

The fact that this colimit exists (assuming that $\mathscr O$ is sifted-complete) is because the $\infty$-category $\Disk_{n/M}$ is sifted (\emph{cf.}~\cite[Proposition 5.5.2.15]{lurie2017higher}).
\end{void}

\begin{rem}
\label{rem-factorization-homology-symmetric-monoidal-functoriality}
By \cite[Proposition 5.5.2.17(2)]{lurie2017higher}, \eqref{eq-factorization-homology} is functorial in $\mathscr O$: Given sifted-complete symmetric monoidal $\infty$-categories $\mathscr O_1$, $\mathscr O_2$ and a symmetric monoidal functor $\mathscr O_1 \rightarrow \mathscr O_2$ commuting with sifted colimits, we have a commutative square
$$
\begin{tikzcd}[column sep = 1.5em]
	\Alg_{\mathbb E_M}(\mathscr O_1) \ar[r, "\int_M"]\ar[d] & \mathscr O_1 \ar[d] \\
	\Alg_{\mathbb E_M}(\mathscr O_2) \ar[r, "\int_M"] & \mathscr O_2
\end{tikzcd}
$$
\end{rem}

\begin{rem}
\label{rem-factorization-homology-symmetric-monoidal}
We shall endow the $\infty$-category $\Alg_{\mathbb E_M}(\mathscr O)$ with the symmetric monoidal structure defined by pointwise tensor product (\emph{cf.}~\cite[Example 3.2.4.4]{lurie2017higher}).

Suppose that $\mathscr O$ is sifted-complete. Then the functor of factorization homology \eqref{eq-factorization-homology} is symmetric monoidal by \cite[Theorem 5.5.3.2]{lurie2017higher}.
\end{rem}

\begin{void}
Let $\mathscr O$ be a sifted-complete symmetric monoidal $\infty$-category. We shall prove that \eqref{eq-factorization-homology} commutes with the \emph{relative} tensor product.

Namely, given an associative algebra $\mathscr A$ in $\Alg_{\mathbb E_M}(\mathscr O)$ and right (respectively, left) $\mathscr A$-module $\mathscr B_1$ (respectively, $\mathscr B_2$), we may form the relative tensor product as geometric realization of the Bar complex (\emph{cf.}~\cite[Definition 4.4.2.10]{lurie2017higher})
$$
\mathscr B_1 \otimes_{\mathscr A} \mathscr B_2 := \colim \mathrm{Bar}_{\mathscr A}(\mathscr B_1, \mathscr B_2)_{\bullet}.
$$
\end{void}

\begin{lem}
\label{lem-factorization-homology-tensor-distribution}
There is a natural isomorphism
\begin{equation}
\label{eq-factorization-homology-tensor-distribution}
	\int_M \mathscr B_1 \otimes_{\mathscr A} \mathscr B_2 \simeq \int_M\mathscr B_1 \otimes_{\int_M \mathscr A} \int_M \mathscr B_2.
\end{equation}
\end{lem}

\begin{proof}
Since \eqref{eq-factorization-homology} is symmetric monoidal (\emph{cf.}~Remark \ref{rem-factorization-homology-symmetric-monoidal}), it suffices to show that \eqref{eq-factorization-homology} commutes with sifted colimits. By construction, it suffices to show that the functor of pre-composition with \eqref{eq-disk-to-em-functor}
$$
\Alg_{\mathbb E_M}(\mathscr O) \rightarrow \Fun(\Disk_{n/M}, \mathscr O)
$$
commutes with sifted colimits.

Since colimits in $\Fun(\Disk_{n/M}, \mathscr O)$ are formed pointwise, this assertion follows from \cite[Proposition 3.2.3.1]{lurie2017higher}.
\end{proof}

\begin{void}
\label{void-factorization-homology-calg}
Finally, we recall the computation of factorization homology of (families of) commutative algebras.

Let $\mathscr O$ be a symmetric monoidal $\infty$-category. Write $\CAlg(\mathscr O)$ for the $\infty$-category of commutative algebras in $\mathscr O$. The pointwise tensor structure on $\CAlg(\mathscr O)$ coincides with the co-Cartesian symmetric monoidal structure (\emph{cf.}~\cite[Proposition 3.2.4.7]{lurie2017higher}).

Since the underlying $\infty$-category of $\mathbb E_M^{\otimes}$ is identified with $\Sing M$ (\emph{cf.}~Remark \ref{rem-tangent-classifier-via-yoneda}), we have a canonical equivalence of $\infty$-categories (\emph{cf.}~\cite[Proposition 2.4.3.9]{lurie2017higher})
\begin{equation}
\label{eq-em-commutative-algebra}
\Fun(\Sing M, \CAlg(\mathscr O)) \simeq \Alg_{\mathbb E_M}(\CAlg(\mathscr O)).
\end{equation}

Given a functor $\mathscr A : \Sing M \rightarrow \CAlg(\mathscr O)$, we shall denote its image under \eqref{eq-em-commutative-algebra} by $\mathscr A_M$.
\end{void}

\begin{lem}
\label{lem-factorization-homology-calg}
Suppose that $\mathscr O$ is cocomplete. Given a functor $\mathscr A : \Sing M \rightarrow \CAlg(\mathscr O)$, there is a canonical isomorphism in $\CAlg(\mathscr O)$:
\begin{equation}
\label{eq-factorization-homology-calg}
\int_M \mathscr A_M \simeq \colim \mathscr A.
\end{equation}
\end{lem}

\begin{proof}
This is a reformulation of (the proof of) \cite[Theorem 5.5.3.8]{lurie2017higher}.
\end{proof}

\subsection{Nonabelian Poincar\'e duality}

\begin{void}
\label{void-compactly-supported-sections}
Denote by $\Spc$ the $\infty$-category of $\infty$-groupoids, endowed with the Cartesian symmetric monoidal structure. Denote by $\Spc_*$ the $\infty$-category of pointed $\infty$-groupoids.

By unstraightening, each functor $\mathscr X : \Sing M \rightarrow \Spc_*$ may be regarded as a Kan fibration over $\Sing M$, endowed with a neutral section. For any $\infty$-groupoid $Y$ over $\Sing M$, we write $\Gamma(Y, \mathscr X)$ for the pointed space $\Maps_{/\Sing M}(Y, \mathscr X)$.

Given a functor $\mathscr X : \Sing M \rightarrow \Spc_*$ and an open subset $U \subset M$, we have the $\infty$-groupoid of \emph{compactly supported sections} of $\mathscr X$ over $U$:
\begin{equation}
\label{eq-compactly-supported-sections}
\Gamma_c(U, \mathscr X) := \colim_{K \subset U} \Gamma(\Sing M, \mathscr X) \times_{\Gamma(\Sing M\setminus K, \mathscr X)} *,
\end{equation}
where the colimit is taken over the poset of compact subsets $K$ of $U$. The expression \eqref{eq-compactly-supported-sections} depends (covariantly) functorially on $U$ and on $\mathscr X$.
\end{void}

\begin{void}
By \cite[Definition 5.5.6.2, Remark 5.5.6.3]{lurie2017higher}, there is a functor of $\infty$-categories
\begin{equation}
\label{eq-functor-local-loop}
	\Omega_M : \Fun(\Sing M, \Spc_*) \rightarrow \Alg_{\mathbb E_M}(\Spc)
\end{equation}
extending \eqref{eq-compactly-supported-sections} in the following sense: Given any $x \in \Sing M$, corresponding to an object $U \rightarrow M$ of the underlying $\infty$-category of $\mathbb E_M^{\otimes}$ (\emph{cf.}~Remark \ref{rem-tangent-classifier-via-yoneda}), the functor $\ev_x$ of taking fiber at $x$ (\emph{cf.}~\S\ref{void-em-algebras}) renders the diagram below commute:
$$
\begin{tikzcd}[column sep = 1.5em]
	\Fun(\Sing M, \Spc_*) \ar[r, "\Omega_M"]\ar[dr, swap, "{\Gamma_c(U, \cdot)}"] & \Alg_{\mathbb E_M}(\Spc) \ar[d, "\ev_x"] \\
	& \Spc
\end{tikzcd}
$$
\end{void}

\begin{lem}
\label{lem-functor-local-loop-finite-limits}
The functor \eqref{eq-functor-local-loop} commutes with finite limits.
\end{lem}

\begin{proof}
The functors $\ev_x$ preserve limits and are jointly conservative when taken over all $x \in \Sing M$. Therefore, it suffices to prove that each functor
$$
\Gamma_c(U, \cdot) : \Fun(\Sing M, \Spc_*) \rightarrow \Spc
$$
preserves finite limits.

This holds because filtered colimits and finite limits commute in $\Spc$.
\end{proof}

\begin{rem}
\label{rem-functor-local-loop-calg}
It follows from Lemma \ref{lem-functor-local-loop-finite-limits} that the functor \eqref{eq-functor-local-loop} is symmetric monoidal with respect to the Cartesian symmetric monoidal structures.

In particular, it induces a functor
\begin{equation}
\label{eq-functor-local-loop-calg}
\Omega_M : \Fun(\Sing M, \CAlg(\Spc)) \rightarrow \CAlg(\Alg_{\mathbb E_M}(\Spc)).
\end{equation}
The target of \eqref{eq-functor-local-loop-calg} is equivalent to $\Alg_{\mathbb E_M}(\CAlg(\Spc))$, as both $\infty$-categories consist of algebra objects over the tensor product $\infty$-operad (\emph{cf.}~\cite[Proposition 2.2.5.6]{lurie2017higher}).
\end{rem}

\begin{void}
Next, we shall recall the statement of nonabelian Poincar\'e duality, due to Salvatore, Segal, and Lurie (\emph{cf.}~\cite{MR1851264, MR2681691, lurie2017higher}).
\end{void}

\begin{prop}[Nonabelian Poincar\'e duality]
\label{prop-poincare-duality}
Let $\mathscr X : \Sing M \rightarrow \Spc_*$ be a functor valued in $n$-connective $\infty$-groupoids. Then there is a canonical isomorphism
\begin{equation}
\label{eq-nonabelian-poincare-duality}
\int_M \Omega_M(\mathscr X) \simeq \Gamma_c(M, \mathscr X).
\end{equation}
\end{prop}

\begin{proof}
This is \cite[Theorem 5.5.6.6]{lurie2017higher}.
\end{proof}

\begin{void}[Local trace map]
\label{void-local-trace-map}
In the remainder of this subsection, we will explain how \eqref{eq-nonabelian-poincare-duality} interacts with the trace maps in \emph{abelian} Poincar\'e duality, when an orientation is provided. For a more complete treatment, see \cite[\S4]{MR4197982}.

Denote by $\integers\Mod$ the stable $\infty$-category of $H\integers$-module spectra. Forgetting the $H\integers$-action and applying connective truncation, we obtain a functor
\begin{equation}
\label{eq-linear-map-forgetful-functor}
\integers\Mod \rightarrow \CAlg(\Spc).
\end{equation}

Let $\mathscr A : \Sing M \rightarrow \integers\Mod$ be a functor. Applying \eqref{eq-functor-local-loop-calg} and \eqref{eq-em-commutative-algebra} to the composition of $\mathscr A$ with \eqref{eq-linear-map-forgetful-functor}, we obtain $\mathbb E_M$-algebras $\Omega_M(\mathscr A)$, respectively $\mathscr A_M$ in $\CAlg(\Spc)$. When $M$ is equipped with a ($\integers$-)orientation, they are related as follows:
\begin{equation}
\label{eq-local-trace-map}
	\tau_M^{\loc} : \Omega_M(\mathscr A) \simeq (\Omega^n\mathscr A)_M,
\end{equation}

We shall refer to \eqref{eq-local-trace-map} as the \emph{local trace map}.
\end{void}

\begin{void}[Construction of \eqref{eq-local-trace-map}]
We shall use a linear version of the functor \eqref{eq-functor-local-loop} (\emph{cf.}~the proof of \cite[Proposition 5.5.6.16]{lurie2017higher}).

Namely, for any stable $\infty$-category $\mathscr O$ admitting limits and colimits, there is a functor
\begin{equation}
\label{eq-functor-local-loop-linear}
\Omega_M : \Fun(\Sing M, \mathscr O) \rightarrow \Alg_{\mathbb E_M}(\mathscr O),
\end{equation}
where $\mathscr O$ is endowed with the Cartesian symmetric monoidal structure. For $\mathscr O := \integers\Mod$, the same-named functors \eqref{eq-functor-local-loop-linear} and \eqref{eq-functor-local-loop-calg} are related by the commutative square
$$
\begin{tikzcd}[column sep = 1.5em]
	\Fun(\Sing M, \integers\Mod) \ar[r, "\Omega_M"]\ar[d, "\eqref{eq-linear-map-forgetful-functor}"] & \Alg_{\mathbb E_M}(\integers\Mod) \ar[d, "\eqref{eq-linear-map-forgetful-functor}"] \\
	\Fun(\Sing M, \CAlg(\Spc)) \ar[r, "\Omega_M"] & \Alg_{\mathbb E_M}(\CAlg(\Spc))
\end{tikzcd}
$$

Since the symmetric monoidal structure on $\mathscr O$ is also co-Cartesian, \eqref{eq-functor-local-loop-linear} may be viewed as an endofunctor of $\Fun(\Sing M, \mathscr O)$ (\emph{cf.}~\cite[Proposition 2.4.3.9]{lurie2017higher}). It remains to identify \eqref{eq-functor-local-loop-linear} with $[-n]$ for $\mathscr O := \integers\Mod$, given an orientation of $M$.

The desired identification is obtained from the natural isomorphisms
\begin{align*}
\Gamma_c(U, \mathscr A) \simeq \Gamma_c(U, \integers[n]) &\otimes \mathscr A_x[-n]  \\
&\simeq \integers \otimes \mathscr A_x[-n] \simeq \mathscr A_x[-n]
\end{align*}
for any $x \in \Sing M$, with corresponding object $U\rightarrow M$ in $\deloop\Top(n)_{/M}$ (\emph{cf.}~Remark \ref{rem-tangent-classifier-via-yoneda}). Here, the identification $\Gamma_c(U, \integers[n]) \simeq \integers$ is provided by the orientation of $M$, which is evidently functorial in $U \rightarrow M$.
\end{void}

\begin{void}
In the context of \S\ref{void-local-trace-map}, we may apply factorization homology \eqref{eq-factorization-homology} to the local trace map \eqref{eq-local-trace-map}. This yields the integrated local trace map $\int_M\tau_M^{\loc}$.

Suppose that $\mathscr A$ is $n$-connective as a $\Spc_*$-valued functor. Then we may apply Lemma \ref{lem-factorization-homology-calg} and Proposition \ref{prop-poincare-duality} to obtain a commutative diagram in $\CAlg(\Spc)$:
\begin{equation}
\label{eq-local-trace-map-integral}
\begin{tikzcd}[column sep = 2em]
	\int_M \Omega_M(\mathscr A) \ar[r, "\int_M\tau_M^{\loc}"]\ar[d, "\eqref{eq-nonabelian-poincare-duality}"] & \int_M  (\Omega^n\mathscr A)_M \ar[d, "\eqref{eq-factorization-homology-calg}"] \\
	\Gamma_c(M, \mathscr A) \ar[r, "\simeq"] & \colim \Omega^n\mathscr A
\end{tikzcd}
\end{equation}
where all arrows are isomorphisms. Here, the fact that \eqref{eq-nonabelian-poincare-duality} lifts to an isomorphism in $\CAlg(\Spc)$ comes from the commutation of $\int_M\Omega_M(\cdot)$ and $\Gamma_c(M, \cdot)$ with finite products, by Lemma \ref{lem-functor-local-loop-finite-limits} and the commutation of sifted colimits with finite products in $\Spc$.
\end{void}

\begin{void}[Global trace map]
\label{void-global-trace-map}
Consider the constant functor $\mathscr A$ with values in $A[k] \in \integers\Mod$, for an abelian group $A$ and an integer $k\ge n$. In this case, we have a map in $\CAlg(\Spc)$:
\begin{equation}
\label{eq-homology-map-to-point}
\colim\Omega^n(\deloop^k A) \simeq \colim \deloop^{k-n} A \rightarrow \deloop^{k-n} A.
\end{equation}

We define the \emph{global trace map} to be the composition of the lower horizontal isomorphism in \eqref{eq-local-trace-map-integral} with \eqref{eq-homology-map-to-point}:
\begin{equation}
\label{eq-global-trace-map}
	\tau_M^{\glob} : \Gamma_c(M, \deloop^k A) \rightarrow \deloop^{k - n} A.
\end{equation}

From the commutative square \eqref{eq-local-trace-map-integral}, we deduce the following compatibility between the local and global trace maps:
\begin{equation}
\label{eq-global-trace-map-compatibility}
\begin{tikzcd}[column sep = 2em]
	\int_M \Omega_M(\deloop^k A) \ar[r, "\int_M\tau_M^{\loc}"]\ar[d, "\eqref{eq-nonabelian-poincare-duality}"] & \int_M (\deloop^{k - n}A)_M \ar[d, "\eqref{eq-homology-map-to-point}\circ\eqref{eq-factorization-homology-calg}"] \\
	\Gamma_c(M, \deloop^k A) \ar[r, "\tau_M^{\glob}"] & \deloop^{k-n}A
\end{tikzcd}
\end{equation}
\end{void}

\subsection{Coefficients}

\begin{void}
\label{void-dgcat}
Denote by $\PrL$ the $\infty$-category of presentable $\infty$-categories with colimit-preserving functors. It admits a symmetric monoidal structure given by the Lurie tensor product.

Denote by $\Vect$ the $\infty$-category of $H\complexes$-module spectra, viewed as a commutative algebra object in $\PrL$. Its module objects in $\PrL$ are called \emph{DG categories}:
$$
\DGCat := \Vect\Mod(\PrL).
$$

Being a module $\infty$-category, $\DGCat$ inherits a symmetric monoidal structure given by the tensor product relative to $\Vect$.
\end{void}

\begin{void}
Given an $\infty$-groupoid $Y$, we define the $\infty$-category of \emph{$\complexes$-local systems} over $Y$ to be the limit of the constant diagram
$$
\LS(Y) := \lim_Y \Vect.
$$

The assignment $Y\mapsto \LS(Y)$ organizes into a functor
\begin{equation}
\label{eq-local-system-contravariant}
	\Spc^{\opposite} \rightarrow \DGCat.
\end{equation}
We denote the image of a morphism $f : Y_1 \rightarrow Y_2$ under \eqref{eq-local-system-contravariant} by $f^{\dagger} : \LS(Y_2) \rightarrow \LS(Y_1)$.

In fact, the functor \eqref{eq-local-system-contravariant} is right adjoint to the functor $\DGCat \rightarrow \Spc^{\opposite}$ sending $\mathscr C$ to $\Maps_{\DGCat}(\mathscr C, \Vect)$. This implies that \eqref{eq-local-system-contravariant} preserves limits.
\end{void}

\begin{rem}
\label{rem-local-systems-as-sheaves}
For each $M \in \Mfd_n$, we may consider the DG category $\Shv_{\{0\}}(M)$ of sheaves of $\complexes$-vector spaces over $M$ whose cohomology sheaves are locally constant. Then there is a canonical equivalence of DG categories
$$
\Shv_{\{0\}}(M) \simeq \LS(\Sing M),
$$
according to \cite[Lemma A.4.2]{MR4368480}.
\end{rem}

\begin{void}
By \cite[\S1.4]{MR4368480}, the DG category $\LS(Y)$ is also \emph{covariantly} functorial in $Y \in \Spc$. More precisely, for each morphism $f : Y_1 \rightarrow Y_1$ in $\Spc$, the functor $f^{\dagger}$ admits a left adjoint $f_{\dagger}$. Thus, we obtain a functor
\begin{equation}
\label{eq-local-system-covariant}
\Spc \rightarrow \DGCat
\end{equation}
by passing to left adjoints in \eqref{eq-local-system-contravariant}.

Moreover, the canonical self-duality of $\Vect$ induces a canonical self-duality of $\LS(Y)$ for each $Y \in \Spc$. Under this self-duality, $f_{\dagger}$ is identified with the dual of $f^{\dagger}$ for every morphism $f$ in $\Spc$. Thus, \eqref{eq-local-system-covariant} may also be obtained from \eqref{eq-local-system-contravariant} by passing to duals and applying the canonical self-duality.
\end{void}

\begin{lem}
\label{lem-local-system-covariant-colimit-commutation}
The functor \eqref{eq-local-system-covariant} commutes with colimits.
\end{lem}

\begin{proof}
Given a diagram $I \rightarrow \Spc$, $i\mapsto Y_i$, the DG category $\colim_{i\in I} \LS(Y_i)$ is dualizable (\emph{cf.}~\cite[Lemma 1.4.8(d)]{MR4368480}). Moreover, its dual is canonically identified with $\lim_{i\in I} \LS(Y_i)$, with transition functors given by pullbacks.

The latter is identified wtih $\LS(\colim_{i\in I}Y_i)$ since \eqref{eq-local-system-contravariant} commutes with limits.
\end{proof}

\begin{void}
Note that \eqref{eq-local-system-contravariant} admits a lax symmetric monoidal structure given by the external tensor product construction.

For $Y_1, Y_2 \in \Spc$, the lax symmetric monoidal structure supplies a functor
\begin{equation}
\label{eq-local-system-tensor-product}
\LS(Y_1) \otimes \LS(Y_2) \rightarrow \LS(Y_1 \times Y_2),\quad \mathscr A_1, \mathscr A_2 \mapsto \mathscr A_1\boxtimes\mathscr A_2.
\end{equation}
\end{void}

\begin{lem}
\label{lem-local-system-kunneth-formula}
The functor \eqref{eq-local-system-contravariant} (hence \eqref{eq-local-system-covariant}) is symmetric monoidal.
\end{lem}

\begin{proof}
It suffices to prove that \eqref{eq-local-system-tensor-product} is an equivalence.

Under the canonical identifications
$$
\LS(Y_1\times Y_2) \simeq \LS(\lim_{Y_2} Y_1) \simeq \lim_{Y_2} \LS(Y_1),
$$
the functor \eqref{eq-local-system-tensor-product} corresponds to \cite[(1.17)]{MR4368480} for $\mathscr C := \LS(Y_1)$ and $Y := Y_2$. Thus the result follows from \cite[Proposition 1.4.10]{MR4368480}.
\end{proof}

\begin{void}[Tautological character]
Finally, we shall define a categorical analogue of the tautological character local system on the classifying space of $\complexes^{\times}$.

Let us view $\complexes^{\times}$ as a discrete abelian group and its deloop $\deloop \complexes^{\times}$ as a commutative algebra in $\Spc$. Since \eqref{eq-local-system-contravariant} is symmetric monoidal (\emph{cf.}~Lemma \ref{lem-local-system-kunneth-formula}), the tautological character local system on $\deloop\complexes^{\times}$ may be viewed as a morphism $\Vect \rightarrow \LS(\deloop\complexes^{\times})$ of cocommutative coalgebras in $\DGCat$.

Dualizing, we obtain a morphism in $\CAlg(\DGCat)$:
\begin{equation}
\label{eq-tautological-character-categorical}
\chi : \LS(\deloop\complexes^{\times}) \rightarrow \Vect,
\end{equation}
which we call the \emph{(categorical) tautological character}.
\end{void}

\section{Local constructions}

In this section, we define the ``quantum torus", or rather its category of representations $\Rep_q(\check T)$ as an $\mathbb E_M$-algebra in $\DGCat$, for an oriented $2$-manifold $M$ (\emph{cf.}~\S\ref{sec-quantum-torus-definition}). The definition is simple, but not quite explicit. In \S\ref{sec-t-structure}-\ref{sec-ribbon-twist}, we make $\Rep_q(\check T)$ more explicit by showing that it is the derived $\infty$-category of its heart $\Rep_q(\check T)^{\heartsuit}$, which may be viewed as a \emph{twisted} family of braided monoidal categories over $M$, and computing its local invariants.

The material of \S\ref{sec-t-structure} and \S\ref{sec-ribbon-twist} will \emph{not} be used later, so the reader only is only interested in the factorization homology of $\Rep_q(\check T)$ may safely skip them.

\subsection{Betti levels}

\begin{void}
\label{void-scheme-to-homotopy-type}
Given a $\complexes$-scheme $Y$ of finite type, we write $Y^{\topology}$ for the topological space underlying the analytification of $Y$ and $\Sing Y^{\topology}$ for its homotopy type (\emph{cf.}~\S\ref{void-category-of-manifolds}).

The association $Y \mapsto \Sing Y^{\topology}$ determines a functor
\begin{equation}
\label{eq-scheme-to-homotopy-type}
\Sing(\cdot)^{\topology} : \Sch_{\ft} \rightarrow \Spc,
\end{equation}
where $\Sch_{\ft}$ denotes the category of $\complexes$-schemes of finite type.
\end{void}

\begin{rem}
\label{rem-scheme-to-homotopy-type-fiber-product}
The functor \eqref{eq-scheme-to-homotopy-type} does \emph{not} commute with finite limits. More precisely, the functor $Y \mapsto Y^{\topology}$ from $\Sch_{\ft}$ to the category of topological spaces commutes with finite limits (\emph{cf.}~\cite[Expos\'e XII, \S1.2]{MR2017446}), but the functor $\Sing$ does not.

However, given morphisms $Y_1 \rightarrow Y \leftarrow Y_2$ in $\Sch_{\ft}$ where $Y_1 \rightarrow Y$ induces a Serre fibration $(Y_1)^{\topology} \rightarrow Y^{\topology}$, then the natural map in $\Spc$ is an isomorphism:
$$
\Sing (Y_1\times_Y Y_2)^{\topology} \simeq \Sing (Y_1)^{\topology} \times_{\Sing Y^{\topology}} \Sing(Y_2)^{\topology}.
$$
\end{rem}

\begin{void}
\label{void-level-general}
Let $X$ be a smooth $\complexes$-curve and $G$ be a reductive group $X$-scheme. We shall define the $\infty$-groupoid $\Level(G)$ of ``Betti levels" of $G$ in this context.

Since $G^{\topology} \rightarrow X^{\topology}$ is a Serre fibration, $\Sing G^{\topology} \rightarrow \Sing X^{\topology}$ inherits a group structure from $G$ (\emph{cf.}~Remark \ref{rem-scheme-to-homotopy-type-fiber-product}). We view $\deloop \Sing G^{\topology}$ as an $\infty$-groupoid over $\Sing X^{\topology}$, equipped with a neutral section $\Sing X^{\topology} \rightarrow \deloop\Sing G^{\topology}$.

Define the $\infty$-groupoid $\Level(G)$ of \emph{Betti levels} of $G$ to be:
\begin{equation}
\label{eq-betti-level-construction}
\Level(G) := \Maps_{\Sing X^{\topology}/}(\deloop\Sing G^{\topology}, \deloop^4\complexes^{\times}),
\end{equation}
where $\deloop^4\complexes^{\times}$ is viewed as an $\infty$-groupoid under $\Sing X^{\topology}$ via the neutral point.
\end{void}

\begin{rem}
Denote by $K(G^{\topology}, 1)$ the classifying space of $G^{\topology}$ relative to $X^{\topology}$.

The definition of $\Level(G)$ renders it a ``categorification" of the reduced cohomology group $H^4_*(K(G^{\topology}, 1), \complexes^{\times})$. Namely, for each $n\ge 0$, we have
$$
\pi_n \Level(G) \simeq H^{4-n}_*(K(G^{\topology}, 1), \complexes^{\times}).
$$

This definition of $\Level(G)$ is inspired by the analogous notions in the \'etale and de Rham cohomological contexts (\emph{cf.}~\cite{MR3769731}, \cite{zhao2022metaplectic}).

To our knowledge, the first authors to suggest that $H^4_*(K(G^{\topology}, 1), \complexes^{\times})$ plays the role of levels for quantum groups are Dijkgraaf and Witten (\emph{cf.}~\cite{MR1048699}).
\end{rem}

\begin{void}
\label{void-exponential-exact-sequence}
Let us now specialize to the case where $G = T$ is an $X$-torus. Denote by $\Lambda$ the smooth $X$-scheme representing the sheaf of cocharacters of $T$.

There is a short exact sequence of topological groups relative to $X^{\topology}$:
\begin{equation}
\label{eq-exponential-exact-sequence}
0 \rightarrow \Lambda^{\topology} \rightarrow \Lie(T)^{\topology} \xrightarrow{\exp} T^{\topology} \rightarrow 1,
\end{equation}
where $\Lie(T)$ denotes the Lie algebra of $T$ and $\exp$ the exponential map. The short exact sequence \eqref{eq-exponential-exact-sequence} identifies $\Sing T^{\topology}$ with $\deloop\Sing \Lambda^{\topology}$. In particular, we obtain
\begin{equation}
\label{eq-level-torus-as-mapping-space}
\Level(T) \simeq \Maps_{\Sing X^{\topology}/}(\deloop^2\Sing\Lambda^{\topology}, \deloop^4\complexes^{\times}).
\end{equation}
\end{void}

\begin{rem}
\label{rem-level-for-torus-as-mapping-space}
The locally constant sheaf of abelian groups $\Lambda^{\topology}$ corresponds to a functor
\begin{equation}
\label{eq-cocharacter-lattice-functor}
\Sing X^{\topology} \rightarrow \integers\Mod.
\end{equation}
To each $x \in \Sing X^{\topology}$, \eqref{eq-cocharacter-lattice-functor} assigns a finite free $\integers$-module, which we view as the fiber of $\Lambda$ at $x$. (This holds literally when $x$ is defined by a $\complexes$-point of $X$.)

Note that the double deloop of \eqref{eq-cocharacter-lattice-functor} classifies $\deloop^2\Sing \Lambda^{\topology}$ in the following sense: Its composition with the forgetful functor $\integers\Mod \rightarrow \Spc$ corresponds to $\deloop^2\Sing\Lambda^{\topology}$ under unstraightening. In particular, $\Level(T)$ is equivalent to the mapping space from the double deloop of \eqref{eq-cocharacter-lattice-functor} to $\deloop^4\complexes^{\times}$ in the $\infty$-category $\Fun(\Sing X^{\topology}, \Spc_*)$.
\end{rem}

\begin{void}[Topological context]
\label{void-topological-context}
In view of Remark \ref{rem-level-for-torus-as-mapping-space}, we find it convenient to change our context from algebra to topology.

Namely, we shall fix an \emph{oriented} $2$-manifold $M$ and a functor $\Lambda : \Sing M \rightarrow \integers\Mod$ valued in finite free $\integers$-modules. By a ``Betti level", we shall mean a morphism
$$
q : \deloop^2\Lambda \rightarrow \deloop^4\complexes^{\times}
$$
in the $\infty$-category $\Fun(\Sing M, \Spc_*)$.

Applications to the algebraic context (\emph{cf.}~\S\ref{void-level-general}) will be obtained by setting $M := X^{\topology}$, $\Lambda :=$ the functor \eqref{eq-cocharacter-lattice-functor}, and $q$ a Betti level in the sense of \S\ref{void-level-general}.
\end{void}

\begin{void}[Digression: cohomology of $\deloop^2\Gamma$]
Let $\Gamma$ be a finite free $\integers$-module.

Write $\Maps_*(\deloop^2\Gamma, \deloop^4\complexes^{\times})$ for the mapping space in $\Spc_*$ and $\Maps_{\integers}(\deloop^2\Gamma, \deloop^4\complexes^{\times})$ for the mapping space in $\integers\Mod$. Write $\Quad(\Gamma, \complexes^{\times})$ for the abelian group
$$
\Quad(\Gamma, \complexes^{\times}) := \Sym^2(\check{\Gamma}) \otimes_{\integers} \complexes^{\times}.
$$
Its elements can be viewed as $\complexes^{\times}$-valued quadratic forms on $\Gamma$: Given $c\otimes \zeta \in (\check{\Gamma})^{\otimes 2}\otimes_{\integers}\complexes^{\times}$, we obtain a quadratic form $\Gamma \rightarrow \complexes^{\times}$ by the expression $\gamma \mapsto \zeta^{c(\gamma, \gamma)}$, and this expression depends only on the symmetrization of $c$.

The $\infty$-groupoid $\Maps_*(\deloop^2\Gamma, \deloop^4\complexes^{\times})$ has nontrivial homotopy groups in degrees $0$ and $2$. Its postnikov tower may be identified explicitly as follows:
\begin{equation}
\label{eq-lattice-deloop-cohomology}
\begin{tikzcd}[column sep = 1em]
	\deloop^2 \pi_2 \Maps_*(\deloop^2 \Gamma, \deloop^4 \complexes^{\times}) \ar[r]\ar[d, "\simeq"] & \Maps_*(\deloop^2\Gamma, \deloop^4\complexes^{\times}) \ar[r]\ar[d, "\simeq"] & \pi_0\Maps_*(\deloop^2\Gamma, \deloop^4\complexes^{\times}) \ar[d, "\simeq"] \\
	\Maps_{\integers}(\deloop^2\Gamma, \deloop^4\complexes^{\times}) \ar[r, "\alpha"] & \Maps_*(\deloop^2\Gamma, \deloop^4\complexes^{\times}) \ar[r, "\beta"] & \Quad(\Gamma, \complexes^{\times})
\end{tikzcd}
\end{equation}
where $\alpha$ coincides with the map induced by the forgetful functor $\integers\Mod \rightarrow \Spc_*$.
\end{void}

\begin{rem}
\label{rem-bisector-splitting}
The map $\beta$ in \eqref{eq-lattice-deloop-cohomology} admits a splitting over the abelian group of $\complexes^{\times}$-valued \emph{bilinear forms} on $\Gamma$:
\begin{equation}
\label{eq-bilinear-form-splitting}
\begin{tikzcd}[column sep = 1em]
	& (\check{\Gamma})^{\otimes 2} \otimes_{\integers} \complexes^{\times} \ar[d, "{c\otimes\zeta \mapsto (\gamma\mapsto \zeta^{c(\gamma, \gamma)})}"]\ar[dl, dotted] \\
	\Maps_*(\deloop^2\Gamma, \deloop^4\complexes^{\times}) \ar[r, "\beta"] & \Quad(\Gamma, \complexes^{\times})
\end{tikzcd}
\end{equation}

Namely, there is a canonical $\integers$-linear map given by cup product
$$
\check{\Gamma}^{\otimes 2} \rightarrow \Maps_*(\deloop^2\Gamma, \deloop^4\integers),\quad y\otimes z \mapsto (\deloop^2 y) \cup (\deloop^2 z),
$$
which defines the dotted arrow in \eqref{eq-bilinear-form-splitting} by tensoring with $\complexes^{\times}$.
\end{rem}

\begin{void}[Quadratic form]
\label{void-betti-level-quadratic-form}
In the context of \S\ref{void-topological-context}, the fiber sequence \eqref{eq-lattice-deloop-cohomology} gives rise to fiber sequences functorial in $x \in \Sing M$:
$$
\Maps_{\integers}(\deloop^2\Lambda_x, \deloop^4\complexes^{\times}) \rightarrow \Maps_*(\deloop^2\Lambda_x, \deloop^4\complexes^{\times}) \rightarrow \Quad(\Lambda_x, \complexes^{\times}),
$$
where $\Lambda_x$ denote the image of $x$ under $\Lambda$.

Taking limit over $\Sing M$, we obtain a fiber sequence
\begin{equation}
\label{eq-level-exact-sequence}
\Maps_{\integers}(\deloop^2\Lambda, \deloop^4\complexes^{\times}) \rightarrow \Maps_*(\deloop^2\Lambda, \deloop^4\complexes^{\times}) \rightarrow \Quad(\Lambda, \complexes^{\times})
\end{equation}
where the first term denotes the mapping space in $\Fun(\Sing M, \integers\Mod)$, the second term the mapping space in $\Fun(\Sing M, \Spc_*)$, and $\Quad(\Lambda, \complexes^{\times})$ consists of locally constant $\complexes^{\times}$-valued quadratic form on $\Lambda$.

In particular, every Betti level $q$ induces a locally constant $\complexes^{\times}$-valued quadratic form $Q$ on $\Lambda$ via the second map of \eqref{eq-level-exact-sequence}. We call $Q$ the \emph{associated quadratic form} of $q$. The fiber sequence \eqref{eq-level-exact-sequence} shows that $Q$ is the obstruction of $q$ to be $\integers$-linear.
\end{void}

\begin{void}[Symmetric form]
\label{void-betti-level-symmetric-form}
To each $\complexes^{\times}$-valued quadratic form $Q$ on a finite free $\integers$-module $\Gamma$, we may associate the symmetric bilinear form
$$
b : \Gamma\otimes \Gamma \rightarrow \complexes^{\times},\quad \gamma_1, \gamma_2\mapsto Q(\gamma_1 + \gamma_2) Q(\gamma_1)^{-1} Q(\gamma_2)^{-1}.
$$
The symmetric form $b$ vanishes if and only if $Q$ is linear. When this happens, $Q$ must take values in $\{\pm 1\}$.

Thus, every Betti level $q$ of $T$ also has an associated symmetric form $b : \Lambda \otimes \Lambda \rightarrow \complexes^{\times}$. We shall see that $b$ is the obstruction of $q$ to be commutative, \emph{i.e.}~$\mathbb E_{\infty}$-monoidal.

Indeed, write $\Maps_{\mathbb E_{\infty}}(\deloop^2 \Lambda, \deloop^4\complexes^{\times})$ for the mapping space in $\Fun(\Sing M, \CAlg(\Spc))$. The forgetful functor defines a map of $\infty$-groupoids
\begin{equation}
\label{eq-betti-level-symmetric-monoidal}
\Maps_{\mathbb E_{\infty}}(\deloop^2 \Lambda, \deloop^4\complexes^{\times}) \rightarrow \Maps_*(\deloop^2 \Lambda, \deloop^4\complexes^{\times})
\end{equation}
\end{void}

\begin{lem}
\label{lem-betti-level-symmetric-monoidal}
\begin{enumerate}
	\item The map \eqref{eq-betti-level-symmetric-monoidal} is fully faithful;
	\item The essential image of \eqref{eq-betti-level-symmetric-monoidal} consists precisely of Betti levels $q$ whose symmetric forms $b$ vanish.
\end{enumerate}
\end{lem}

\begin{proof}
This is \cite[Remark 4.6.7]{MR3769731}.
\end{proof}

\subsection{The $\mathbb E_M$-category $\Rep_q(\check T)$}
\label{sec-quantum-torus-definition}

\begin{void}
We remain in the context of \S\ref{void-topological-context} and fix a Betti level
\begin{equation}
\label{eq-torus-betti-level}
q : \deloop^2 \Lambda \rightarrow \deloop^4\complexes^{\times}.
\end{equation}

In this subsection, we construct the \emph{quantum torus} as an $\mathbb E_M$-algebra (\emph{cf.}~\S\ref{void-em-algebras}) in the symmetric monoidal $\infty$-category $\DGCat$ (\emph{cf.}~\S\ref{void-dgcat}):
\begin{equation}
\label{eq-quantum-torus}
\Rep_q(\check T) \in \Alg_{\mathbb E_M}(\DGCat).
\end{equation}
\end{void}

\begin{void}
We first recall a special case of a construction in \ref{void-factorization-homology-calg}: The equivalence \eqref{eq-em-commutative-algebra}, when restricted to \emph{constant} functors, yields a functor
\begin{equation}
\label{eq-calg-to-ex-algebra}
(\cdot)_M : \CAlg(\mathscr O) \rightarrow \CAlg(\Alg_{\mathbb E_M}(\mathscr O))
\end{equation}
for any symmetric monoidal $\infty$-category $\mathscr O$. (We identified the target with $\Alg_{\mathbb E_M}(\CAlg(\mathscr O))$, \emph{cf.}~Remark \ref{rem-functor-local-loop-calg}.)

The functor \eqref{eq-calg-to-ex-algebra} is easy to describe: It is the composition of the canonical equivalence $\CAlg(\mathscr O)\simeq\CAlg(\CAlg(\mathscr O))$ (\emph{cf.}~\cite[Example 3.2.4.5]{lurie2017higher}) with the forgetful functor.
\end{void}

\begin{void}[Construction of \eqref{eq-quantum-torus}]
\label{void-quantum-torus-construction}
Denote by $\mathscr H_q$ the fiber of \eqref{eq-torus-betti-level}, viewed as a functor $\Sing M \rightarrow \Spc_*$ equipped with an action of $\deloop^3\complexes^{\times} \in \CAlg(\Spc)$.

Applying the functor \eqref{eq-functor-local-loop}, we obtain
$$
\mathscr H^{\loc}_q := \Omega_M(\mathscr H_q) \in \Alg_{\mathbb E_M}(\Spc)
$$
equipped with an action of the commutative algebra (\emph{cf.}~Remark \ref{rem-functor-local-loop-calg})
$$
\Omega_M(\deloop^3\complexes^{\times}) \in \CAlg(\Alg_{\mathbb E_M}(\Spc)).
$$

We shall identify $\Omega_M(\deloop^3\complexes^{\times})$ with $(\deloop \complexes^{\times})_M$ via the local trace map \eqref{eq-local-trace-map}. (This identification depends on the orientation of $M$.)

On the other hand, the tautological character \eqref{eq-tautological-character-categorical} induces a morphism
\begin{equation}
\label{eq-tautological-character-ex}
\LS((\deloop\complexes^{\times})_M) \simeq \LS((\deloop\complexes^{\times}))_M\xrightarrow{\chi_M} \Vect_M
\end{equation}
in $\CAlg(\Alg_{\mathbb E_M}(\DGCat))$, where the isomorphism uses the symmetric monoidal structure on $\LS$ (\emph{cf.}~Lemma \ref{lem-local-system-kunneth-formula}) and the second map uses functoriality of \eqref{eq-calg-to-ex-algebra}.

We define the $\mathbb E_M$-algebra in $\DGCat$:
\begin{equation}
\label{eq-quantum-torus-definition}
\Rep_q(\check T) := \LS(\mathscr H_q^{\loc}) \otimes_{\LS((\deloop\complexes^{\times})_M)} \Vect_M
\end{equation}
where $\LS(\mathscr H_q^{\loc})$ is acted on by $\LS((\deloop\complexes^{\times})_M)$ via the $(\deloop\complexes^{\times})_M$-action on $\mathscr H_q^{\loc}$, and $\Vect_M$ is acted on by $\LS((\deloop\complexes^{\times})_M)$ via the tautological character \eqref{eq-tautological-character-ex}.
\end{void}

\begin{rem}
Suppose that $q$ is trivial. Then we have $\mathscr H_q \simeq \deloop^3\complexes^{\times} \times \deloop^2\Lambda$. In particular, $\mathscr H_q$ lifts to a functor $\Sing M \rightarrow \CAlg(\Spc)$. Applying $\Omega_M$ and using the local trace map \eqref{eq-local-trace-map}, we obtain an identification
$$
\mathscr H_q^{\loc} \simeq (\deloop\complexes^{\times})_M \times \Lambda_M.
$$

This shows that $\Rep_q(\check T)$ is identified with $\LS(\Lambda)_M$, \emph{i.e.}~the $\mathbb E_M$-algebra associated to the functor $\LS(\Lambda) : \Sing M \rightarrow \CAlg(\DGCat)$ under \eqref{eq-em-commutative-algebra}. The latter corresponds to the locally constant sheaf $\Rep(\check T)$ of representations of the torus $\check T$ with sheaf of \emph{characters} $\Lambda$.
\end{rem}

\begin{rem}
\label{rem-quantum-torus-limit-presentation}
The formula \eqref{eq-quantum-torus-definition} expresses $\Rep_q(\check T)$ as the geometric realization of the Bar complex associated to the $\LS((\deloop\complexes^{\times})_M)$-modules $\LS(\mathscr H_q^{\loc})$ and $\Vect_M$.

By passing to right adjoints, we may also realize $\Rep_q(\check T)$ as the totalization of the co-Bar complex associated to the $\LS((\deloop\complexes^{\times})_M)$-comodules $\LS(\mathscr H_q^{\loc})$ and $\Vect_M$, defined using the pullback functoriality \eqref{eq-local-system-contravariant} of local systems:
\begin{align}
\notag
	\Rep_q(\check T) &\simeq \colim \mathrm{Bar}_{\LS((\deloop\complexes^{\times})_M)} (\LS(\mathscr H_q^{\loc}), \Vect_M) \\
\label{eq-quantum-torus-limit-presentation}
	&\simeq \lim \coBar_{\LS((\deloop\complexes^{\times})_M)} (\LS(\mathscr H_q^{\loc}), \Vect_M).
\end{align}
\end{rem}

\begin{void}[Fibers of $\Rep_q(\check T)$]
\label{void-quantum-torus-fiber}
The limit presentation \eqref{eq-quantum-torus-limit-presentation} gives a concrete description of the fiber of $\Rep_q(\check T)$ at any $x \in \Sing M$.

Indeed, the fiber of $\mathscr H_q^{\loc}$ at $x$ is a $\complexes^{\times}$-gerbe over the set $\Lambda_x$:
$$
\begin{tikzcd}
	\mathscr H_{q, x}^{\loc} \ar[d, "\deloop\complexes^{\times}"] \\
	\Lambda_x
\end{tikzcd}
$$
The expression \eqref{eq-quantum-torus-limit-presentation} identifies $\Rep_q(\check T)_x$ as the DG category of local systems on $\mathscr H_{q, x}^{\loc}$ which are $\deloop\complexes^{\times}$-equivariant against the tautological character local system.

In particular, $\Rep_q(\check T)_x$ admits a $\Lambda_x$-grading determined by the support:
$$
\Rep_q(\check T)_x \simeq \bigoplus_{\lambda \in \Lambda_x} \Rep_q(\check T)_x^{\lambda},
$$
where each summand $\Rep_q(\check T)_x^{\lambda}$ is \emph{non-canonically} equivalent to $\Vect$.
\end{void}

\begin{rem}
Denote by $b$ the symmetric form associated to $q$ (\emph{cf.}~\S\ref{void-betti-level-symmetric-form}).

When $b$ vanishes, $q$ lifts to a morphism in $\Fun(\Sing M, \CAlg(\Spc))$ (\emph{cf.}~Lemma \ref{lem-betti-level-symmetric-monoidal}). This equips $\mathscr H_q^{\loc}$ with the structure of a commutative algebra in $\Alg_{\mathbb E_M}(\Spc)$, so $\Rep_q(\check T)$ also lifts to a commutative algebra in $\Alg_{\mathbb E_M}(\DGCat)$.

In \S\ref{sec-ribbon-twist}, we shall establish a converse to this assertion: When $b$ is nonvanishing at a point $x \in \Sing M$, the fiber $\Rep_q(\check T)_x$ does \emph{not} lift to a commutative algebra.
\end{rem}

\subsection{$t$-structure}
\label{sec-t-structure}

\begin{void}
We remain in the context of \S\ref{void-topological-context}.

In this subsection, we shall explain in what sense $\Rep_q(\check T)$ is the derived category of its heart compatibly with the $\mathbb E_M$-algebra structure.

Denote by $\Vect^{\le 0} \subset \Vect$ the full subcategory of connective objects. Denote by $\Vect^{\heartsuit}$ the heart of $\Vect^{\le 0}$, \emph{i.e.}~the abelian category of $\complexes$-vector spaces.
\end{void}

\begin{void}
The $\infty$-categories $\Vect^{\le 0}$ and $\Vect^{\heartsuit}$ inherit symmetric monoidal structures from $\Vect$, so we may view them as objects of $\CAlg(\PrL)$.

Consider the morphisms in $\CAlg(\PrL)$:
$$
\Vect^{\heartsuit} \leftarrow \Vect^{\le 0} \rightarrow \Vect.
$$
They induce symmetric monoidal functors
\begin{align}
\label{eq-discrete-object-functor}
	\Vect^{\le 0}\Mod(\PrL) \rightarrow \Vect^{\heartsuit}\Mod(\PrL), \quad& \mathscr C\mapsto\mathscr C\otimes_{\Vect^{\le 0}}\Vect^{\heartsuit}, \\
\label{eq-spectrum-object-functor}
	\Vect^{\le 0}\Mod(\PrL) \rightarrow  \DGCat,\quad& \mathscr C\mapsto \mathscr C\otimes_{\Vect^{\le 0}}\Vect.
\end{align}
\end{void}

\begin{rem}
Since $\Vect^{\heartsuit}$ is identified with $\Vect^{\le 0}\otimes \Set$, the functor \eqref{eq-discrete-object-functor} coincides with the functor of taking discrete objects (\emph{cf.}~\cite[Example 4.8.1.22]{lurie2017higher}).

Since $\Vect$ is identified with $\Vect^{\le 0} \otimes \Sptr$, the functor \eqref{eq-spectrum-object-functor} coincides with the functor of taking spectrum objects (\emph{cf.}~\cite[Example 4.8.1.23]{lurie2017higher}).
\end{rem}

\begin{void}
Given $Y \in \Spc$, the DG category $\LS(Y)$ carries a $t$-structure, with connective and coconnective parts defined by
$$
\LS(Y)^{\le 0} := \lim_Y \Vect^{\le 0},\quad \LS(Y)^{\ge 0} := \lim_Y \Vect^{\ge 0}.
$$

Note that $f^{\dagger}$ is $t$-exact. In particular, its left adjoint $f_{\dagger}$ is right $t$-exact. This implies that the functor \eqref{eq-local-system-covariant} restricts to a functor
\begin{equation}
\label{eq-local-system-connective}
	\Spc \rightarrow \Vect^{\le 0}\Mod(\PrL), \quad Y \mapsto \LS(Y)^{\le 0}.
\end{equation}

Furthermore, \eqref{eq-local-system-connective} inherits a symmetric monoidal structure from \eqref{eq-local-system-covariant} (\emph{cf.}~Lemma \ref{eq-local-system-connective}) and we recover \eqref{eq-local-system-covariant} as the composition of \eqref{eq-local-system-connective} with the functor of taking spectrum objects \eqref{eq-spectrum-object-functor}.
\end{void}

\begin{void}[The heart of $\Rep_q(\check T)$]
\label{void-quantum-torus-heart}
The construction of $\Rep_q(\check T)$ (\emph{cf.}~\S\ref{void-quantum-torus-construction}) may thus be repeated with \eqref{eq-local-system-connective} instead of \eqref{eq-local-system-covariant}. This yields an $\mathbb E_M$-algebra in $\Vect^{\le 0}\Mod(\PrL)$
$$
\Rep_q(\check T)^{\le 0} := \LS(\mathscr H_q^{\loc})^{\le 0} \otimes_{\LS((\deloop\complexes^{\times})_M)^{\le 0}} (\Vect^{\le 0})_M
$$
together with an isomorphism of $\mathbb E_M$-algebras in $\DGCat$:
\begin{equation}
\label{eq-quantum-torus-as-spectrum-object}
\Rep_q(\check T) \simeq \Rep_q(\check T)^{\le 0} \otimes_{\Vect^{\le 0}} \Vect.
\end{equation}

Likewise, we may form the symmetric monoidal functor 
\begin{equation}
\label{eq-local-system-heart}
\Spc \rightarrow \Vect^{\heartsuit}\Mod(\PrL),\quad Y \mapsto \LS(Y)^{\heartsuit}
\end{equation}
by composing \eqref{eq-local-system-connective} with the functor of taking discrete objects \eqref{eq-discrete-object-functor}. The construction of \S\ref{void-quantum-torus-construction} then yields an $\mathbb E_M$-algebra in $\Vect^{\heartsuit}\Mod(\PrL)$
$$
\Rep_q(\check T)^{\heartsuit} := \LS(\mathscr H_q^{\loc})^{\heartsuit} \otimes_{\LS((\deloop\complexes^{\times})_M)^{\heartsuit}} (\Vect^{\heartsuit})_M,
$$
together with an isomorphism of such:
$$
\Rep_q(\check T)^{\heartsuit} \simeq \Rep_q(\check T)^{\le 0}\otimes_{\Vect^{\le 0}} \Vect^{\heartsuit}.
$$
\end{void}

\begin{rem}
The observation of Remark \ref{rem-quantum-torus-limit-presentation} also applies to $\Rep_q(\check T)^{\le 0}$ and $\Rep_q(\check T)^{\heartsuit}$. Namely, by passing to right adjoints, we obtain equivalences
\begin{align}
\label{eq-quantum-torus-connective-limit}
	\Rep_q(\check T)^{\le 0} & \simeq \lim \coBar_{\LS((\deloop\complexes^{\times})_M)^{\le 0}}(\LS(\mathscr H_q^{\loc})^{\le 0}, (\Vect^{\le 0})_M), \\
\label{eq-quantum-torus-heart-limit}
	\Rep_q(\check T)^{\heartsuit} & \simeq \lim \coBar_{\LS((\deloop\complexes^{\times})_M)^{\heartsuit}}(\LS(\mathscr H_q^{\loc})^{\heartsuit}, (\Vect^{\heartsuit})_M).
\end{align}

As in \S\ref{void-quantum-torus-fiber}, the expression \eqref{eq-quantum-torus-heart-limit} allows us to identify the fiber $\Rep_q(\check T)^{\heartsuit}_x$ at $x \in \Sing M$ as the \emph{abelian} category of local systems on $\mathscr H_{q, x}^{\loc}$ which are $\deloop\complexes^{\times}$-equivariant against the tautological character local system.
\end{rem}

\begin{void}
Next, we shall show that $\Rep_q(\check T)^{\le 0}$ (hence $\Rep_q(\check T)$, by \eqref{eq-quantum-torus-as-spectrum-object}) is completely determined by $\Rep_q(\check T)^{\heartsuit}$ as an $\mathbb E_M$-algebra. This requires some formalism of derived $\infty$-categories established by Lurie (\emph{cf.}~\cite[Appendix C]{lurie2018spectral})

Denote by $\Groth_{\infty}^{\sep}$ the $1$-full subcategory of $\PrL$ whose objects are separated Grothendieck prestable $\infty$-categories and whose morphisms are exact functors.\footnote{In \emph{op.cit.}, this $\infty$-category is denoted by $\Groth_{\infty}^{\mathrm{lex}, \sep}$.}

Denote by $\Groth_0$ the $1$-full subcategory of $\PrL$ whose objects are Grothendieck \emph{abelian} categories and whose morphisms are exact functors. (By definition, morphisms in $\Groth_{\infty}^{\sep}$ and $\Groth_0$ commute with colimits.)

By \cite[Theorem C.5.4.9]{lurie2018spectral}, the functor of taking discrete object
\begin{equation}
\label{eq-grothendieck-category-heart}
(\cdot)^{\heartsuit} : \Groth_{\infty}^{\sep} \rightarrow \Groth_0
\end{equation}
admits a left adjoint, whose value at $A \in \Groth_0$ is identified with the connective part of the derived $\infty$-category $D(A)^{\le 0}$ (\emph{cf.}~\cite[Proposition C.5.3.2, Proposition C.5.4.5]{lurie2018spectral}).
\end{void}

\begin{void}
By \cite[Theorem C.5.4.16]{lurie2018spectral}, $\Groth_0$ inherits a symmetric monoidal structure from $\PrL$. We endow $\Groth_{\infty}^{\sep}$ with the symmetric monoidal structure given by the \emph{separated} Lurie tensor product (\emph{cf.}~\cite[Corollary C.4.6.2]{lurie2018spectral}).\footnote{For both assertions, we invoked the fact that exactness of functors in $\PrL$ is preserved by tensor product (\emph{cf.}~\cite[Proposition C.4.4.1]{lurie2018spectral}).}

Consider $\Vect^{\le 0}$ (respectively $\Vect^{\heartsuit}$) as a commutative algebra in $\Groth_{\infty}^{\sep}$ (respectively, $\Groth_0$). Since \eqref{eq-grothendieck-category-heart} is symmetric monoidal, it lifts to a functor
\begin{equation}
\label{eq-grothendieck-category-linear-heart}
(\cdot)^{\heartsuit} : \Vect^{\le 0}\Mod(\Groth_{\infty}^{\sep}) \rightarrow \Vect^{\heartsuit}\Mod(\Groth_0).
\end{equation}
Since \eqref{eq-grothendieck-category-heart} admits a left adjoint, so does \eqref{eq-grothendieck-category-linear-heart}.
\end{void}

\begin{void}
Since $\LS(Y)$ is left complete (in particular, left separated) for any $Y \in \Spc$ and the pullback functor $f^{\dagger} : \LS(Y_2) \rightarrow \LS(Y_1)$, for any morphism $f : Y_1 \rightarrow Y_2$ in $\Spc$, is $t$-exact, the functor \eqref{eq-local-system-contravariant} induces functors
\begin{align*}
	\Spc^{\opposite} \rightarrow \Vect^{\le 0}\Mod(\Groth_{\infty}^{\sep}), &\quad Y \mapsto \LS(Y)^{\le 0} \\
	\Spc^{\opposite} \rightarrow \Vect^{\heartsuit}\Mod(\Groth_0), &\quad Y \mapsto \LS(Y)^{\heartsuit}
\end{align*}

The expressions \eqref{eq-quantum-torus-connective-limit}, \eqref{eq-quantum-torus-heart-limit} define $\Rep_q(\check T)^{\le 0}$, respectively $\Rep_q(\check T)^{\heartsuit}$ as $\mathbb E_M$-algebras in $\Vect^{\le 0}\Mod(\Groth_{\infty}^{\sep})$, respectively $\Vect^{\heartsuit}\Mod(\Groth_0)$.
\end{void}

\begin{void}
Denote by $L$ the left adjoint of \eqref{eq-grothendieck-category-linear-heart}. Being the left adjoint of a symmetric monoidal functor, $L$ is \emph{oplax} symmetric monoidal.

We do not know if $L$ is symmetric monoidal in general. Thus, we do not know if the image $L(\mathscr A)$ of an $\mathbb E_M$-algebra $\mathscr A$ in $\Vect^{\heartsuit}\Mod(\Groth_0)$ automatically carries an $\mathbb E_M$-algebra structure. This \emph{does} happen, however, if the leftward arrow in the correspondence
$$
\bigotimes_{i\in I} L(\mathscr A(U_i)) \leftarrow L(\bigotimes_{i\in I} \mathscr A(U_i)) \rightarrow L(\mathscr A(U))
$$
associated to any active morphism $\{U_i\}_{i\in I} \rightarrow U$ of $\mathbb E_M^{\otimes}$ is an isomorphism.

In particular, if $\mathscr A$ is \emph{fiberwise semisimple}, \emph{i.e.}~$\mathscr A_x$ is semisimple for every $x\in\Sing M$, then $L(\mathscr A)$ carries a natural $\mathbb E_M$-algebra structure.

Since $\Rep_q(\check T)^{\heartsuit}$ is fiberwise semisimple (\emph{cf.}~\S\ref{void-quantum-torus-fiber}), we obtain
$$
L(\Rep_q(\check T)^{\heartsuit}) \in \Alg_{\mathbb E_M}(\Vect^{\le 0}\Mod(\Groth_{\infty}^{\sep})).
$$
\end{void}

\begin{lem}
\label{lem-quantum-torus-connective-derived-category}
There is a natural isomorphism in $\Alg_{\mathbb E_M}(\Vect^{\le 0}\Mod(\Groth_{\infty}^{\sep}))$:
\begin{equation}
\label{eq-quantum-torus-connective-derived-category}
L(\Rep_q(\check T)^{\heartsuit}) \simeq \Rep_q(\check T)^{\le 0}.
\end{equation}
\end{lem}

\begin{proof}
The morphism in one direction is defined by the counit of the adjunction between $L$ and \eqref{eq-grothendieck-category-linear-heart}.

The fact that it is an equivalence may be checked fiberwise, where it follows immediately from the description of \S\ref{void-quantum-torus-fiber}.
\end{proof}

\begin{rem}
Informally, Lemma \ref{lem-quantum-torus-connective-derived-category} and \eqref{eq-quantum-torus-as-spectrum-object} express the fact that $\Rep_q(\check T)$ is the derived $\infty$-category of $\Rep_q(\check T)^{\heartsuit}$ ``compatibly with the $\mathbb E_M$-algebra structures."

Namely, for each $x\in \Sing M$, the DG category $\Rep_q(\check T)_x$ is equivalent to the derived $\infty$-category $D(\Rep_q(\check T)_x^{\heartsuit})$, and the $\mathbb E_M$-algebra structure on $\Rep_q(\check T)$ is determined by the $\mathbb E_M$-algebra structure on $\Rep_q(\check T)^{\heartsuit}$.
\end{rem}

\subsection{The ribbon structure}
\label{sec-ribbon-twist}

\begin{void}
We remain in the context of \S\ref{void-topological-context}.

The goal of this subsection is to provide interpretations of the discrete invariants $Q$ and $b$ of $q$ (\emph{cf.}~\S\ref{void-betti-level-quadratic-form}, \S\ref{void-betti-level-symmetric-form}) in terms of the $\mathbb E_M$-algebra $\Rep_q(\check T)^{\heartsuit}$ (\emph{cf.}~\S\ref{void-quantum-torus-heart}).

Informally, $\Rep_q(\check T)^{\heartsuit}$ may be viewed as a family of ``twisted" braided monoidal categories parametrized by $M$. Our results say that $b$ controls the square of the commutativity constraint on $\Rep_q(\check T)^{\heartsuit}$, while $Q$ controls an additional ribbon structure (\emph{cf.}~Proposition \ref{prop-square-of-braiding}, Proposition \ref{prop-ribbon-twist}).
\end{void}

\begin{void}
\label{void-fiber-twisted-operad}
Given an object $\tau \in \deloop\Top(2)$, we define the \emph{$\tau$-twisted $\mathbb E_2$-operad} to be
$$
\mathbb E_{\tau}^{\otimes} := \deloop\Top(2)^{\otimes} \times_{\deloop\Top(2)^{\sqcup}} *^{\sqcup},
$$
using the map $*^{\sqcup} \rightarrow \deloop\Top(2)^{\sqcup}$ induced from $\tau$.

Any \emph{neutralization} of $\tau$, \emph{i.e.}~isomorphism with the neutral point of $\deloop\Top(2)$, induces an isomorphism between $\mathbb E_{\tau}^{\otimes}$ and the operad $\mathbb E_2^{\otimes}$ (\emph{cf.}~\cite[Example 5.4.2.15]{lurie2017higher}).

Given $x\in \Sing M$, the map $x : * \rightarrow \Sing M$ induces a morphism of $\infty$-operads
\begin{equation}
\label{eq-fiber-twisted-operad}
\mathbb E_{\tau_x}^{\otimes} \rightarrow \mathbb E_M^{\otimes},
\end{equation}
where $\tau_x$ denotes the image of $x$ under the map $\tau_M : \Sing M \rightarrow \deloop\Top(2)$ classifying the tangent microbundle of $M$ (\emph{cf.}~Remark \ref{rem-tangent-classifier-via-yoneda}).

In particular, given any symmetric monoidal $\infty$-category $\mathscr O$ and $\mathscr A \in \Alg_{\mathbb E_M}(\mathscr O)$, restriction along \eqref{eq-fiber-twisted-operad} defines $\mathscr A_x \in \Alg_{\mathbb E_{\tau_x}}(\mathscr O)$, whose underlying object of $\mathscr O$ is the fiber of $\mathscr A$ at $x$.
\end{void}

\begin{void}
Recall that $\mathbb E_2$-algebras in the $\infty$-category $\Cat_0$ of (1-)categories are precisely braided monoidal categories (\emph{cf.}~\cite[Example 5.1.2.4]{lurie2017higher}).

Thus, given any $x \in \Sing M$ and a neutralization of $\tau_x$, the fiber $\Rep_q(\check T)^{\heartsuit}_x$ admits the structure of a braided monoidal category.

Given two objects $V^{\lambda_1}$, $V^{\lambda_2}$ of $\Rep_q(\check T)^{\heartsuit}_x$ with gradings $\lambda_1, \lambda_2 \in \Lambda_x$ (\emph{cf.}~\S\ref{void-quantum-torus-fiber}), the square of the commutativity constraint
\begin{equation}
\label{eq-square-of-braiding}
V^{\lambda_1} \otimes V^{\lambda_2} \simeq V^{\lambda_2} \otimes V^{\lambda_1} \simeq V^{\lambda_1} \otimes V^{\lambda_2}
\end{equation}
is an automorphism of $V^{\lambda_1} \otimes V^{\lambda_2}$.
\end{void}

\begin{prop}
\label{prop-square-of-braiding}
The automorphism \eqref{eq-square-of-braiding} equals multiplication by $b(\lambda_1, \lambda_2)$.
\end{prop}

\begin{void}
\label{void-ribbon-structure-context}
We shall prove Proposition \ref{prop-square-of-braiding} along with an assertion describing a ``twisted" ribbon structure on $\Rep_q(\check T)^{\heartsuit}_x$ in terms of the quadratic form $Q$.

We fix a smooth structure on $M$. This allows us to reduce the tangent microbundle of $M$ to its tangent bundle, classified by a map $\tau_M : \Sing M \rightarrow \deloop \GL_2^+$, where $\GL_2^+$ is the group of orientation-preserving automorphisms of $\reals^2$.

Consider the $\infty$-operad (\emph{cf.}~\cite[Example 5.4.2.16]{lurie2017higher})
$$
\mathbb E_{\deloop\GL_2^+}^{\otimes} := \deloop\Top(2)^{\otimes} \times_{\deloop\Top(2)^{\sqcup}} (\deloop\GL_2^+)^{\sqcup}.
$$
For any $x \in \Sing M$, the map $\tau_x : * \rightarrow \deloop\GL_2^+$ induces a morphism of $\infty$-operads
\begin{equation}
\label{eq-twisted-e2-to-oriented}
\mathbb E_{\tau_x}^{\otimes} \rightarrow \mathbb E_{\deloop\GL_2^+}^{\otimes}.
\end{equation}

We shall lift the $\mathbb E_{\tau_x}$-algebra $\Rep_q(\check T)_x$ to an $\mathbb E_{\deloop\GL_2^+}$-algebra along \eqref{eq-twisted-e2-to-oriented}.
\end{void}

\begin{rem}
\label{rem-twisted-oriented-e2-assembly}
Denote by $T_x$ the oriented $2$-dimensional vector space classified by $\tau_x \in \deloop\GL_2^+$. Then its orientation-preserving automorphisms form a topological group $\GL^+(T_x)$ and we have a canonical isomorphism of $\infty$-groupoids over $\deloop\Top(2)$:
$$
\tau_x / \GL^+(T_x) \simeq \deloop\GL_2^+.
$$

Given a symmetric monoidal $\infty$-category $\mathscr O$, we write $\Alg_{\mathbb E_{\tau_x}}(\mathscr O)^{\GL^+(T_x)}$ for the $\infty$-category of $\GL^+(T_x)$-invariants of $\Alg_{\mathbb E_{\tau_x}}(\mathscr O)$. Since $\mathbb E_{\tau_x}^{\otimes}$ defines a family of $\infty$-operads over $\deloop\GL^+(T_x)$ with assembly is $\mathbb E_{\deloop\GL_2^+}^{\otimes}$ (\emph{cf.}~\cite[Remark 5.4.2.13]{lurie2017higher}), we obtain
$$
\Alg_{\mathbb E_{\tau_x}}(\mathscr O)^{\GL^+(T_x)} \simeq \Alg_{\mathbb E_{\deloop\GL_2^+}}(\mathscr O).
$$
\end{rem}

\begin{void}[Construction of the $\mathbb E_{\deloop\GL_2^+}$-algebra structure]
\label{void-ribbon-structure-construction}
Applying the construction of the quantum torus (\emph{cf.}~\S\ref{void-quantum-torus-construction}) with $T_x$ instead of $M$ and the \emph{constant} Betti level $q_x : \deloop^2\Lambda_x \rightarrow \deloop^4\complexes^{\times}$ instead of $q$, we obtain
\begin{equation}
\label{eq-quantum-torus-tangent}
\Rep_{q_x}(\check T) \in \Alg_{\mathbb E_{T_x}}(\DGCat)
\end{equation}
which recovers $\Rep_q(\check T)_x$ under the canonical isomorphism $\mathbb E_{\tau_x}^{\otimes} \simeq \mathbb E_{T_x}^{\otimes}$ of $\infty$-operads.

To lift $\Rep_q(\check T)_x$ to $\Alg_{\mathbb E_{\deloop\GL_2^+}}(\DGCat)$, it thus suffices to endow $\Rep_{q_x}(\check T)$ with a $\GL^+(T_x)$-equivariance structure (\emph{cf.}~Remark \ref{rem-twisted-oriented-e2-assembly}).

Consider the commutative diagram of $\mathbb E_{T_x}$-algebras in $\Spc$:
\begin{equation}
\label{eq-heisenberg-algebra-tangent}
\begin{tikzcd}[column sep = 1.5em]
	\Omega_{T_x}(\deloop^2\Lambda_x) \ar[r, "\simeq"] \ar[d, "\Omega_{T_x}(q_x)"] & (\Lambda_x)_{T_x} \ar[d] \\
	\Omega_{T_x}(\deloop^4\complexes^{\times}) \ar[r, "\simeq"] & (\deloop^2\complexes^{\times})_{T_x}
\end{tikzcd}
\end{equation}
where the horizontal isomorphisms are given by the local trace maps (\emph{cf.}~\S\ref{void-local-trace-map}). By construction, it suffices to endow the right vertical map in \eqref{eq-heisenberg-algebra-tangent} with a $\GL^+(T_x)$-equivariance structure, with respect to the natural $\GL^+(T_x)$-equivariance on $(\Lambda_x)_{T_x}$ and $(\deloop^2\complexes^{\times})_{T_x}$.

Note that the left vertical arrow of \eqref{eq-heisenberg-algebra-tangent} admits an $\GL^+(T_x)$-equivariance structure by functoriality with respect to $T_x$. The desired structure follows because the local trace map \eqref{eq-local-trace-map} for $T_x$ is naturally $\GL^+(T_x)$-equivariant.
\end{void}

\begin{void}
\label{void-ribbon-twist}
We return to the context of \S\ref{void-ribbon-structure-context}.

Since the underlying $\infty$-category of $\mathbb E_{\deloop\GL_2^+}$ is $\deloop\GL_2^+$, the $\mathbb E_{\deloop\GL_2^+}$-algebra structure on $\Rep_q(\check T)_x$ yields a functor
\begin{equation}
\label{eq-quantum-torus-fiber-ribbon}
\deloop\GL_2^+ \rightarrow \DGCat,
\end{equation}
sending the point $\tau_x$ (\emph{not} the neutral point!) to $\Rep_q(\check T)_x$.

Let us identify $\deloop \GL_2^+$ with $\deloop^2\integers$ as objects of $\Spc_*$. This identification is determined by the homotopy equivalences
$$
\deloop\integers \simeq S^1 \simeq \SO(2) \simeq \GL_2^+.
$$

Once a neutralization of $\tau_x \in \deloop\GL_2^+$ is chosen, the generator $1\in\integers$ defines an automorphism $\theta$ of the identity endofunctor on $\Rep_q(\check T)_x$ under \eqref{eq-quantum-torus-fiber-ribbon}. Given $V^{\lambda} \in \Rep_q(\check T)_x^{\heartsuit}$ with grading $\lambda \in \Lambda_x$, the automorphism $\theta$ specializes to an automorphism
\begin{equation}
\label{eq-ribbon-twist}
\theta_{V^{\lambda}} : V^{\lambda} \simeq V^{\lambda}.
\end{equation}
\end{void}

\begin{prop}
\label{prop-ribbon-twist}
The automorphism \eqref{eq-ribbon-twist} equals multiplication by $Q(\lambda)$.
\end{prop}

\begin{void}
\label{void-fundamental-class-tangent-space}
Before we prove Proposition \ref{prop-ribbon-twist}, we relate the local trace map (\emph{cf.}~\S\ref{void-local-trace-map}) for $\reals^2$ to the $\SO(2)$-action. Namely, consider the (inverse of the) local trace map
\begin{equation}
\label{eq-fundamental-class-tangent-space}
	\integers \simeq \Gamma_c(\reals^2, \deloop^2\integers),\quad 1\mapsto [\reals^2].
\end{equation}

Writing $B_r \subset \reals^2$ for the closed disk of radius $r\in\reals_{>0}$ centered at the origin, we may express $\Gamma_c(\reals^2, \deloop^2\integers)$ as the fiber of the map
\begin{equation}
\label{eq-compactly-supported-section-tangent-space}
\Gamma(\reals^2, \deloop^2\integers) \rightarrow \colim_{r\rightarrow\infty} \Gamma(\reals^2\setminus B_r, \deloop^2\integers).
\end{equation}
The inclusion $\reals^2\setminus B_r \subset \reals^2$ is $\SO(2)$-equivariant. Taking the quotient by $\SO(2)$ and passing to homotopy types, it gives rise to the neutral map $e : * \rightarrow \deloop\SO(2)$. In particular, the fiber of \eqref{eq-compactly-supported-section-tangent-space} is identified with the fiber of $e^* : \Maps(\deloop\SO(2), \deloop^2\integers) \rightarrow \deloop^2\integers$ via pullback:
\begin{equation}
\label{eq-compactly-supported-section-tangent-space-so2}
\Maps_*(\deloop\SO(2), \deloop^2\integers) \simeq \Gamma_c(\reals^2, \deloop^2\integers).
\end{equation}

Under the isomorphism \eqref{eq-compactly-supported-section-tangent-space-so2}, the class $[\reals^2]$ corresponds to the canonical isomorphism $\deloop\SO(2) \simeq \deloop^2\integers$ in $\Spc_*$ (\emph{cf.}~\S\ref{void-ribbon-twist}).
\end{void}

\begin{void}
We now turn to the proof of Proposition \ref{prop-ribbon-twist}.
\end{void}

\begin{proof}[Proof of Proposition \ref{prop-ribbon-twist}]
We identify $T_x$ with $\reals^2$ using the chosen neutralization of $\tau_x$.

The commutative square \eqref{eq-heisenberg-algebra-tangent} specializes to a commutative square of $\infty$-groupoids endowed with $\SO(2)$-action:
\begin{equation}
\label{eq-heisenberg-algebra-tangent-global}
\begin{tikzcd}[column sep = 1.5em]
	\Gamma_c(\reals^2, \deloop^2\Lambda_x) \ar[r, "\simeq"]\ar[d, "{\Gamma_c(\reals^2, q_x)}"] & \Lambda_x \ar[d] \\
	\Gamma_c(\reals^2, \deloop^4\complexes^{\times}) \ar[r, "\simeq"] & \deloop^2\complexes^{\times}
\end{tikzcd}
\end{equation}
Here, the horizontal isomorphisms are the local trace maps. The group $\SO(2)$ acts on the left column of \eqref{eq-heisenberg-algebra-tangent-global} via its action on $\reals^2$ and acts trivially on the right column.

Let us express $\theta_{V^{\lambda}}$ in terms of the $\SO(2)$-equivariant morphism $\Lambda_x \rightarrow \deloop^2\complexes^{\times}$ in \eqref{eq-heisenberg-algebra-tangent-global}. Indeed, evaluating the latter at $\lambda$ yields an $\SO(2)$-invariant object of $\deloop^2\complexes^{\times}$, \emph{i.e.}~an object of the $\infty$-groupoid
\begin{equation}
\label{eq-heisenberg-algebra-tangent-global-so2}
\Gamma_c(\reals^2, q_x)(\lambda) \in \Maps(\deloop\SO(2), \deloop^2\complexes^{\times})
\end{equation}
Under the identification $\deloop\SO(2) \simeq \deloop^2\integers$ (\emph{cf.}~\S\ref{void-ribbon-twist}), the class of \eqref{eq-heisenberg-algebra-tangent-global-so2} is an element $\theta(\lambda) \in \complexes^{\times}$. By construction, the automorphism $\theta_{V^{\lambda}}$ acts as multiplication by $\theta(\lambda)$.

It remains to prove the following equality for each $\lambda \in \Lambda_x$:
\begin{equation}
\label{eq-ribbon-twist-via-quadratic-form}
	\theta(\lambda) = Q(\lambda).
\end{equation}

First, we observe that both sides of \eqref{eq-ribbon-twist-via-quadratic-form} depend linearly on $q$: For $\theta(\lambda)$, this holds because \eqref{eq-heisenberg-algebra-tangent-global-so2} depends linearly on $q$, while for $Q(\lambda)$, this holds because \eqref{eq-lattice-deloop-cohomology} comes from a fiber sequence in $\integers\Mod$. By Remark \ref{rem-bisector-splitting}, we may assume that $q$ is of the form
$$
q \simeq (\deloop^2 y) \cup (\deloop^2 z),
$$
for characters $y : \Lambda_x \rightarrow \integers$ and $z : \Lambda_x \rightarrow \complexes^{\times}$. The value $Q(\lambda)$ then equals $z(\lambda)^{y(\lambda)}$.

Let us determine \eqref{eq-heisenberg-algebra-tangent-global-so2} for this choice of $q$. Indeed, as an $\SO(2)$-invariant object of $\Gamma_c(\reals^2, \deloop^4\complexes^{\times})$, it is given by the cup product
\begin{equation}
\label{eq-cup-product-compactly-supported-section}
(y(\lambda) \cdot [\reals^2]) \cup (z(\lambda) \cdot [\reals^2]).
\end{equation}
We may view $z(\lambda) \cdot [\reals^2]$ as an $\SO(2)$-invariant object of $\Gamma(\reals^2, \deloop^2\complexes^{\times})$ (\emph{i.e.}~forgetting that it is compactly supported) and consequently as an $\SO(2)$-equivariant morphism
\begin{equation}
\label{eq-character-as-linear-morphism}
z(\lambda) \cdot [\reals^2] : \integers \rightarrow \deloop^2\complexes^{\times}.
\end{equation}

If we forget the $\SO(2)$-equivariance structure on \eqref{eq-character-as-linear-morphism}, then it is simply the $H\integers$-linear morphism $z(\lambda) : \integers \rightarrow \deloop^2\complexes^{\times}$. The $\SO(2)$-equivariance structure, however, is determined by $[\reals^2]$: It says that the image of the generator $1 \in \integers$ is an $\SO(2)$-invariant object of $\deloop^2\complexes^{\times}$ whose class equals $z(\lambda)$ (\emph{cf.}~\S\ref{void-fundamental-class-tangent-space}).

The cup product \eqref{eq-cup-product-compactly-supported-section} is isomorphic to a Yoneda product, \emph{i.e.}~the image of the $\SO(2)$-invariant object $y(\lambda) \cdot [\reals^2] \in \Gamma_c(\reals^2, \deloop^2\integers)$ under the double deloop of \eqref{eq-character-as-linear-morphism}. By naturality of the local trace map (\emph{cf.}~\S\ref{void-local-trace-map}), we have a commutative square
$$
\begin{tikzcd}[column sep = 1em]
	\Gamma_c(\reals^2, \deloop^2\integers) \ar[d, "{z(\lambda)\cdot[\reals^2]}"]\ar[r, "\simeq"] & \integers\ar[d, "{z(\lambda)\cdot[\reals^2]}"] \\
	\Gamma_c(\reals^2, \deloop^4\complexes^{\times}) \ar[r, "\simeq"] & \deloop^2\complexes^{\times}
\end{tikzcd}
$$
Thus, \eqref{eq-heisenberg-algebra-tangent-global-so2} is isomorphic to the image of $y(\lambda) \in \integers$ under the $\SO(2)$-equivariant morphism \eqref{eq-character-as-linear-morphism}. In particular, its class equals $z(\lambda)^{y(\lambda)}$, as desired.
\end{proof}

\begin{void}
Finally, we shall deduce Proposition \ref{prop-square-of-braiding} from Proposition \ref{prop-ribbon-twist} and standard facts about braided monoidal categories.
\end{void}

\begin{proof}[Proof of Proposition \ref{prop-square-of-braiding}]
Choose a smooth structure on $M$ and a neutralization of $\tau_x$ as a point of $\deloop\GL_2^+$.

The construction of \S\ref{void-ribbon-structure-construction} lifts the braided monoidal category $\Rep_q(\check T)_x^{\heartsuit}$ lifts to an $\mathbb E_{\deloop\GL_2^+}$-algebra in $\Cat_0$. The latter is precisely a ribbon structure on $\Rep_q(\check T)_x^{\heartsuit}$, whose ribbon twist is provided by the automorphism \eqref{eq-ribbon-twist} (\emph{cf.}~\cite[\S4]{salvatore2001frameddiscsoperadsequivariant}).

In particular, this implies that the automorphism \eqref{eq-square-of-braiding} equals
$$
\theta_{V^{\lambda_1}\otimes V^{\lambda_2}} \circ (\theta_{V^{\lambda_1}}^{-1} \otimes \theta_{V^{\lambda_2}}^{-1}).
$$
By Proposition \ref{prop-ribbon-twist}, this is the multiplication by $b(\lambda_1, \lambda_2)$, as desired.
\end{proof}

\begin{rem}
In the proof of Proposition \ref{prop-square-of-braiding}, we used the term ``ribbon" as in \cite[Definition 4.9]{salvatore2001frameddiscsoperadsequivariant}, referring only to additional structure of the twist.

Some authors call this structure ``balanced" and reserve the term ``ribbon" for \emph{rigid} balanced braided monoidal categories. The subcategory of \emph{compact objects} in $\Rep_q(\check T)_x^{\heartsuit}$ is indeed rigid, hence ``ribbon" in the stronger sense.
\end{rem}

\section{Global constructions}

In this section, we calculate the factorization homology of $\Rep_q(\check T)$ over an oriented $2$-manifold $M$ (\emph{cf.}~Theorem \ref{thm-factorization-homology-quantum-torus}) and use it to prove a version of the quantum Betti geometric Langlands conjecture for tori (\emph{cf.}~Corollary \ref{cor-ben-zvi-nadler}).

\subsection{Calculation of $\int_M \Rep_q(\check T)$}

\begin{void}
Let $M$ be an oriented $2$-manifold. Let $\Lambda : \Sing M \rightarrow \integers\Mod$ taking values in finite free $\integers$-modules and fix a morphism in $\Fun(\Sing M, \Spc_*)$:
\begin{equation}
\label{eq-betti-level-factorization-homology}
q : \deloop^2\Lambda \rightarrow \deloop^4\complexes^{\times}.
\end{equation}

In this context, we have defined the $\mathbb E_M$-algebra $\Rep_q(\check T)$ in $\DGCat$ (\emph{cf.}~\S\ref{void-quantum-torus-construction}). The goal of this subsection is to compute the factorization homology of $\Rep_q(\check T)$ (\emph{cf.}~\S\ref{void-factorization-homology-construction}).
\end{void}

\begin{void}
\label{void-betti-level-global-section}
Applying the functor $\Gamma_c(M, \cdot)$ of compactly supported sections (\emph{cf.}~\S\ref{void-compactly-supported-sections}) to \eqref{eq-betti-level-factorization-homology}, we obtain a morphism in $\Spc$:
\begin{equation}
\label{eq-betti-level-global-section}
\Gamma_c(M, q) : \Gamma_c(M, \deloop^2\Lambda) \rightarrow \Gamma_c(M, \deloop^4\complexes^{\times}).
\end{equation}

Composing \eqref{eq-betti-level-global-section} with the global trace map $\tau_M^{\glob}$ (\emph{cf.}~\S\ref{void-global-trace-map}, for $A := \complexes^{\times}$ and $k := 4$), we obtain a morphism in $\Spc$:
\begin{equation}
\label{eq-global-gerbe}
\Gamma_c(M, \deloop^2\Lambda) \rightarrow \deloop^2\complexes^{\times}
\end{equation}

Denote by $\mathscr H_q^{\glob}$ the fiber of \eqref{eq-global-gerbe}. Thus $\mathscr H_q^{\glob}$ is an $\infty$-groupoid equipped with a $\deloop\complexes^{\times}$-action. In particular, $\LS(\mathscr H_q^{\glob})$ carries an action of $\LS(\deloop\complexes^{\times})$, using the symmetric monoidal structure on $\LS$ (\emph{cf.}~Lemma \ref{lem-local-system-kunneth-formula}). Viewing $\Vect$ as an $\LS(\deloop\complexes^{\times})$-module via the tautological character $\chi$ \eqref{eq-tautological-character-categorical}, we form
$$
\LS_q(\Gamma_c(M, \deloop^2\Lambda)) := \LS(\mathscr H_q^{\glob}) \otimes_{\LS(\deloop\complexes^{\times})} \Vect.
$$
\end{void}

\begin{thm}
\label{thm-factorization-homology-quantum-torus}
There is a canonical equivalence in $\DGCat$:
$$
\int_M \Rep_q(\check T) \simeq \LS_q(\Gamma_c(M, \deloop^2\Lambda)).
$$
\end{thm}

\begin{proof}
Recall that $\Rep_q(\check T)$ is defined as the tensor product $\LS(\mathscr H_q^{\loc}) \otimes_{\LS((\deloop\complexes^{\times})_M)} \Vect_M$, where $\mathscr H_q^{\loc}$ is defined as $\Omega_M(\mathscr H_q)$ for $\mathscr H_q$ the fiber of \eqref{eq-betti-level-factorization-homology} (\emph{cf.}~\S\ref{void-quantum-torus-construction}). Let us rewrite it as the tensor product
\begin{equation}
\label{eq-quantum-torus-omega-expression}
\Rep_q(\check T) \simeq \LS(\Omega_M(\mathscr H_q)) \otimes_{\LS(\Omega_M(\deloop^3\complexes^{\times}))} \Vect_M,
\end{equation}
where $\Vect_M$ is viewed as a $\LS(\Omega_M(\deloop^3\complexes^{\times}))$-module via the composition
\begin{equation}
\label{eq-local-section-character}
\LS(\Omega_M(\deloop^3\complexes^{\times})) \simeq \LS((\deloop\complexes^{\times})_M) \xrightarrow{\chi_M} \Vect_M,
\end{equation}
where the isomorphism is $\LS(\tau_M^{\loc})$ for the local trace map $\tau_M^{\loc}$ (\emph{cf.}~\S\ref{void-local-trace-map}).

Using the presentation \eqref{eq-quantum-torus-omega-expression}, we compute:
\begin{align*}
	\int_M \Rep_q(\check T) & \simeq \int_M \LS(\Omega_M(\mathscr H_q)) \otimes_{\int_M\LS(\Omega_M(\deloop^3\complexes^{\times}))} \int_M\Vect_M & \text{(Lemma \ref{lem-factorization-homology-tensor-distribution})} \\
	& \simeq \LS(\int_M \Omega_M(\mathscr H_q)) \otimes_{\LS(\int_M \Omega_M(\deloop^3\complexes^{\times}))} \Vect & \text{(Remark \ref{rem-factorization-homology-symmetric-monoidal-functoriality})} \\
	& \simeq \LS(\Gamma_c(M, \mathscr H_q)) \otimes_{\LS(\Gamma_c(M, \deloop^3\complexes^{\times}))} \Vect & \text{(Proposition \ref{prop-poincare-duality})}
\end{align*}
Here, the identification $\int_M\Vect_M \simeq \Vect$ follows from the symmetric monoidal structure on $\int_M$ (\emph{cf.}~Remark \ref{rem-factorization-homology-symmetric-monoidal}).

\emph{Claim}: The $\LS(\Gamma_c(M, \deloop^3\complexes^{\times}))$-module structure on $\Vect$, appearing in the above expression, is induced from the morphism in $\CAlg(\DGCat)$:
\begin{equation}
\label{eq-global-section-character}
\LS(\Gamma_c(M, \deloop^3\complexes^{\times})) \rightarrow \LS(\deloop\complexes^{\times}) \xrightarrow{\chi} \Vect
\end{equation}
where the first morphism is $\LS(\tau_M^{\glob})$ for the global trace map $\tau_M^{\glob}$ (\emph{cf.}~\S\ref{void-global-trace-map}).

Indeed, assuming the claim, we obtain the desired isomorphism:
\begin{align*}
\int_M\Rep_q(\check T) & \simeq \LS(\Gamma_c(M, \mathscr H_q)) \otimes_{\LS(\Gamma_c(M, \deloop^3\complexes^{\times}))} \LS(\deloop\complexes^{\times}) \otimes_{\LS(\deloop\complexes^{\times})} \Vect \\
& \simeq \LS(\mathscr H_q^{\glob}) \otimes_{\LS(\deloop\complexes^{\times})} \Vect \\
& \simeq \LS_q(\Gamma_c(M, \deloop^2\Lambda)),
\end{align*}
using the fact that $\mathscr H_q^{\glob}$ is the quotient of $\Gamma_c(M, \mathscr H_q) \times \deloop\complexes^{\times}$ by the anti-diagonal action of $\Gamma_c(M, \deloop^3\complexes^{\times})$ and that $\LS$ commutes with colimits (\emph{cf.}~Lemma \ref{lem-local-system-covariant-colimit-commutation}).

To prove the claim, it suffices to identify \eqref{eq-global-section-character} with the factorization homology of \eqref{eq-local-section-character} under nonabelian Poincar\'e duality (\emph{cf.}~Proposition \ref{prop-poincare-duality}). This amounts to the solid commutative diagram in $\CAlg(\DGCat)$ below:
\begin{equation}
\label{eq-local-global-section-character}
\begin{tikzcd}[column sep = 3em]
	\int_M\LS(\Omega_M(\deloop^2\complexes^{\times})) \ar[r, "\int_M\LS(\tau_M^{\loc})"]\ar[d, "\simeq"] & \int_M \LS((\deloop\complexes^{\times})_M) \ar[d, dotted] \ar[r, "\int_M \chi_M"] & \int_M \Vect_M \ar[d, "\simeq"] \\
	\LS(\Gamma_c(M, \deloop^3\complexes^{\times})) \ar[r, "\LS(\tau_M^{\glob})"] & \LS(\deloop\complexes^{\times}) \ar[r, "\chi"] & \Vect
\end{tikzcd}
\end{equation}

We shall supply the dotted arrow in \eqref{eq-local-global-section-character} making both squares commute. Indeed, we let it be the composite
\begin{equation}
\label{eq-local-global-section-character-dotted-arrow}
\int_M \LS(\deloop\complexes^{\times})_M \simeq \colim \LS(\deloop\complexes^{\times}) \rightarrow \LS(\deloop\complexes^{\times}),
\end{equation}
where the isomorphism is \eqref{eq-factorization-homology-calg} (applied to $\mathscr A := \LS(\deloop\complexes^{\times})$) and the second map is induced from the identity on $\LS(\deloop\complexes^{\times})$, as it determines a constant functor out of $\Sing M$.

The right square of \eqref{eq-local-global-section-character} commutes by naturality of the construction of \eqref{eq-local-global-section-character-dotted-arrow} with respect to $\chi$. The left square of \eqref{eq-local-global-section-character-dotted-arrow} commutes because it is the image of \eqref{eq-global-trace-map-compatibility} under $\LS$.
\end{proof}

\subsection{The Ben-Zvi--Nadler conjecture for $T$}

\begin{void}
Let $X$ be a smooth $\complexes$-curve, assumed projective and connected. Let $T$ be an $X$-torus and $q$ be a Betti level for $T$ in the sense of \S\ref{void-level-general}.

Our goal is to interpret Theorem \ref{thm-factorization-homology-quantum-torus} as a version of a conjecture of Ben-Zvi and Nadler (\emph{cf.}~\cite[Conjecture 4.27]{MR3821166}) for $T$.

To do so, we invoke the passage from the algebraic to the topological context (\emph{cf.}~\S\ref{void-topological-context}): We regard $\Lambda$ as a functor $\Sing X^{\topology} \rightarrow \integers\Mod$ and $q$ as a morphism $\deloop^2\Lambda \rightarrow \deloop^4\complexes^{\times}$ in the $\infty$-category $\Fun(\Sing X^{\topology}, \Spc_*)$. To ease the notation, we shall denote factorization homology over $X^{\topology}$ by $\int_X$.
\end{void}

\begin{void}
We also need to extend the construction of the underlying homotopy type of a $\complexes$-scheme of finite type (\emph{cf.}~\S\ref{void-scheme-to-homotopy-type}) to $\complexes$-prestacks.

Indeed, we write $\PStk(\Sch_{\ft})$ for the $\infty$-category of functors $(\Sch_{\ft})^{\opposite} \rightarrow \Spc$ and consider the left Kan extension of \eqref{eq-scheme-to-homotopy-type} along the Yoneda embedding. This yields a functor
\begin{equation}
\label{eq-prestack-to-homotopy-type}
\PStk(\Sch_{\ft}) \rightarrow \Spc,
\end{equation}
which we will still denote by $\Sing(\cdot)^{\topology}$.
\end{void}

\begin{rem}
\label{rem-stack-homotopy-type-descent}
Let $\mathscr Y$ be an algebraic stack with a smooth cover $f : Z \rightarrow \mathscr Y$ with $Z \in \Sch_{\ft}$ such that for any $S \in \Sch_{\ft}$ over $\mathscr Y$ the base change $f_S : Z\times_{\mathscr Y}S \rightarrow S$ induces a Serre fibration $(Z\times_{\mathscr Y} S)^{\topology} \rightarrow S^{\topology}$, then we have a canonical isomorphism in $\Spc$:
\begin{equation}
\label{eq-stack-homotopy-type-descent}
\colim \Sing (Z^{\bullet}_{/\mathscr Y})^{\topology} \simeq \Sing \mathscr Y^{\topology},
\end{equation}
where $Z^{\bullet}_{/\mathscr Y}$ denotes the \v{C}ech nerve of $f$.

To see this, we note that \eqref{eq-prestack-to-homotopy-type} commutes with colimits, so it suffices to prove that given a morphism $Z \rightarrow Y$ in $\Sch_{\ft}$ inducing a Serre fibration $Z^{\topology} \rightarrow Y^{\topology}$, its \v{C}ech nerve induces an isomorphism
$$
\colim \Sing (Z^{\bullet}_{/Y})^{\topology} \simeq \Sing Y^{\topology}.
$$
By Remark \ref{rem-scheme-to-homotopy-type-fiber-product}, $\Sing(Z^{\bullet}_{/Y})^{\topology}$ is identified with the \v{C}ech nerve of the effective epimorphism $\Sing Z^{\topology} \rightarrow \Sing Y^{\topology}$, so this follows from \cite[Corollary 6.2.3.5]{MR2522659}.

Notably, for the algebraic stack $\deloop T$ with the smooth cover $X\rightarrow \deloop T$, \eqref{eq-stack-homotopy-type-descent} yields an isomorphism in $\Spc$:
\begin{equation}
\label{eq-quotient-torus-commutes-with-homotopy-type}
\deloop \Sing T^{\topology} \simeq \Sing (\deloop T)^{\topology}.
\end{equation}
\end{rem}

\begin{void}
Denote by $\Bun_T$ the moduli stack of $T$-bundles over $X$, whose $S$-points (for $S\in\Sch_{\ft}$) are $T$-bundles over $X\times S$. Let us construct a morphism in $\Spc$:
\begin{equation}
\label{eq-bunt-homotopy-type}
	\Sing(\Bun_T)^{\topology} \rightarrow \Gamma(\Sing X^{\topology}, \deloop^2\Lambda).
\end{equation}

Indeed, applying $\Sing(\cdot)^{\topology}$ to the universal $T$-bundle
$$
X \times \Bun_T \rightarrow \deloop T
$$
and using its commutation with finite products (\emph{cf.}~Remark \ref{rem-scheme-to-homotopy-type-fiber-product}), we obtain a morphism:
\begin{align}
\notag
\Sing X^{\topology} \times \Sing(\Bun_T)^{\topology} &\rightarrow \Sing(\deloop T)^{\topology} \\
\label{eq-universal-bundle-homotopy-type}
& \simeq \deloop\Sing T^{\topology} \simeq \deloop^2 \Lambda,
\end{align}
where the two isomorphism are \eqref{eq-quotient-torus-commutes-with-homotopy-type}, respectively the one of \S\ref{void-exponential-exact-sequence}. The desired map \eqref{eq-bunt-homotopy-type} now follows from adjunction.
\end{void}

\begin{lem}
The morphism \eqref{eq-bunt-homotopy-type} is an isomorphism.
\end{lem}

\begin{proof}
It suffices to prove that \eqref{eq-bunt-homotopy-type} induces an isomorphism on homotopy groups.

The connected components of $\Sing(\Bun_T)^{\topology}$ are classified by the first Chern class of a $T$-bundle, thus in natural bijection with $H^2(X^{\topology}, \Lambda)$.

Denote by $\Bun_T^0 \subset \Bun_T$ the substack of $T$-bundles with vanishing first Chern class. Then the homotopy groups of $\Sing(\Bun_T^0)^{\topology}$ are computed by Atiyah--Bott uniformization, which yields $H^1(X^{\topology}, \Lambda)$, respectively $H^0(X^{\topology}, \Lambda)$ in degrees $1$ and $2$.
\end{proof}

\begin{void}
By post-composing \eqref{eq-universal-bundle-homotopy-type} with the Betti level $q$, we obtain a map $\Sing(\Bun_T)^{\topology} \rightarrow \Gamma(\Sing X^{\topology}, \deloop^4\complexes^{\times})$ by adjunction.

Post-composing with the global trace map $\tau^{\glob}_M$ (\emph{cf.}~\S\ref{void-global-trace-map}, for $M := X^{\topology}$, $A := \complexes^{\times}$ and $k := 4$), we obtain a map in $\Spc$:
\begin{equation}
\label{eq-bunt-betti-gerbe}
\Sing(\Bun_T)^{\topology} \rightarrow \deloop^2\complexes^{\times},
\end{equation}
which can be viewed as a Betti $\complexes^{\times}$-gerbe over $\Bun_T$.

By construction, \eqref{eq-global-gerbe} and \eqref{eq-bunt-betti-gerbe} are intertwined by the isomorphism \eqref{eq-bunt-homotopy-type}:
\begin{equation}
\label{eq-bunt-gerbe-compatibility}
\begin{tikzcd}[column sep = 1.5em]
	\Sing(\Bun_T)^{\topology} \ar[r, "\eqref{eq-bunt-betti-gerbe}"]\ar[d, "\simeq"] & \deloop^2\complexes^{\times} \ar[d, "\simeq"] \\
	\Gamma(\Sing X^{\topology}, \deloop^2\Lambda) \ar[r, "\eqref{eq-global-gerbe}"] & \deloop^2\complexes^{\times}
\end{tikzcd}
\end{equation}
\end{void}

\begin{void}
The authors of \cite{MR3821166} propose to study the $\infty$-category $\Shv_{\mathscr N, q}(\Bun_T)$ of sheaves of $\complexes$-vector space on $\Bun_T$ twisted by the $\complexes^{\times}$-gerbe \eqref{eq-bunt-betti-gerbe}, satisfying the condition of having \emph{nilpotent} singular support.

Since the global nilpotent cone of $\Bun_T$ is the zero section, this condition is equivalent to having locally constant cohomology sheaves. By Remark \ref{rem-local-systems-as-sheaves}, we may identify
$$
\Shv_{\mathscr N, q}(\Bun_T) \simeq \LS_q(\Bun_T),
$$
where $\LS_q(\Bun_T)$ denotes the $\infty$-category of local systems over $\Sing(\Bun_T)^{\topology}$ twisted by the $\complexes^{\times}$-gerbe \eqref{eq-bunt-betti-gerbe}. In other words, we have an equivalence
\begin{equation}
\label{eq-bunt-twisted-local-systems}
\Shv_{\mathscr N, q}(\Bun_T) \simeq \LS(\mathscr G_q^{\glob}) \otimes_{\LS(\deloop\complexes^{\times})} \Vect,
\end{equation}
where $\mathscr G_q^{\glob}$ denotes the fiber of \eqref{eq-bunt-betti-gerbe}, endowed with the natural $\deloop\complexes^{\times}$-action, and $\Vect$ is viewed as an $\LS(\deloop\complexes^{\times})$-module via the tautological character $\chi$.

The following is a version of \cite[Conjecture 4.27]{MR3821166} for $X$-tori.
\end{void}

\begin{cor}
\label{cor-ben-zvi-nadler}
There is a canonical equivalence in $\DGCat$:
$$
\int_X \Rep_q(\check T) \simeq \Shv_{\mathscr N, q}(\Bun_T).
$$
\end{cor}

\begin{proof}
Using \eqref{eq-bunt-twisted-local-systems}, we reduce to constructing an equivalence
\begin{equation}
\label{eq-ben-zvi-nadler-tensor}
\int_X \Rep_q(\check T) \simeq \LS(\mathscr G_q^{\glob}) \otimes_{\LS(\deloop\complexes^{\times})} \Vect.
\end{equation}

In view of \eqref{eq-bunt-gerbe-compatibility}, we may replace $\mathscr G_q^{\glob}$ by the fiber $\mathscr H_q^{\glob}$ of \eqref{eq-global-gerbe}. Thus, the right-hand-side of \eqref{eq-ben-zvi-nadler-tensor} is canonically equivalent to $\LS(\mathscr H_q^{\glob}) \otimes_{\LS(\deloop\complexes^{\times})} \Vect$, which is by definition the DG category $\LS_q(\Gamma(\Sing X^{\topology}, \deloop^2\Lambda))$ (\emph{cf.}~\S\ref{void-betti-level-global-section}).

The desired equivalence \eqref{eq-ben-zvi-nadler-tensor} is now provided by Theorem \ref{thm-factorization-homology-quantum-torus}.
\end{proof}

\bibliographystyle{amsalpha}
\bibliography{bibliography.bib}

\providecommand{\bysame}{\leavevmode\hbox to3em{\hrulefill}\thinspace}
\providecommand{\MR}{\relax\ifhmode\unskip\space\fi MR }
\providecommand{\MRhref}[2]{%
  \href{http://www.ams.org/mathscinet-getitem?mr=#1}{#2}
}
\providecommand{\href}[2]{#2}
\begin{thebibliography}{GKRV22}

\bibitem[AF15]{MR3431668}
David Ayala and John Francis, \emph{Factorization homology of topological
  manifolds}, J. Topol. \textbf{8} (2015), no.~4, 1045--1084. \MR{3431668}

\bibitem[AF20]{MR4197982}
\bysame, \emph{A factorization homology primer}, Handbook of homotopy theory,
  CRC Press/Chapman Hall Handb. Math. Ser., CRC Press, Boca Raton, FL, [2020]
  \copyright 2020, pp.~39--101. \MR{4197982}

\bibitem[BZBJ18]{MR3847209}
David Ben-Zvi, Adrien Brochier, and David Jordan, \emph{Integrating quantum
  groups over surfaces}, J. Topol. \textbf{11} (2018), no.~4, 874--917.
  \MR{3847209}

\bibitem[BZN18]{MR3821166}
David Ben-Zvi and David Nadler, \emph{Betti geometric {L}anglands}, Algebraic
  geometry: {S}alt {L}ake {C}ity 2015, Proc. Sympos. Pure Math., vol.~97, Amer.
  Math. Soc., Providence, RI, 2018, pp.~3--41. \MR{3821166}

\bibitem[DW90]{MR1048699}
Robbert Dijkgraaf and Edward Witten, \emph{Topological gauge theories and group
  cohomology}, Comm. Math. Phys. \textbf{129} (1990), no.~2, 393--429.
  \MR{1048699}

\bibitem[GKRV22]{MR4368480}
Dennis Gaitsgory, David Kazhdan, Nick Rozenblyum, and Yakov Varshavsky, \emph{A
  toy model for the {D}rinfeld--{L}afforgue shtuka construction}, Indag. Math.
  (N.S.) \textbf{33} (2022), no.~1, 39--189. \MR{4368480}

\bibitem[GL18]{MR3769731}
Dennis Gaitsgory and Sergey Lysenko, \emph{Parameters and duality for the
  metaplectic geometric {L}anglands theory}, Selecta Math. (N.S.) \textbf{24}
  (2018), no.~1, 227--301, References are to the corrected version, available
  at: \url{https://lysenko.perso.math.cnrs.fr}. \MR{3769731}

\bibitem[GR03]{MR2017446}
Alexander Grothendieck and Mich\`ele Raynaud, \emph{Rev\^{e}tements \'{e}tales
  et groupe fondamental ({SGA} 1)}, Documents Math\'{e}matiques (Paris)
  [Mathematical Documents (Paris)], vol.~3, Soci\'{e}t\'{e} Math\'{e}matique de
  France, Paris, 2003, S\'{e}minaire de g\'{e}om\'{e}trie alg\'{e}brique du
  Bois Marie 1960--61. [Algebraic Geometry Seminar of Bois Marie 1960-61],
  Directed by A. Grothendieck, With two papers by M. Raynaud, Updated and
  annotated reprint of the 1971 original [Lecture Notes in Math., 224,
  Springer, Berlin; MR0354651 (50 \#7129)]. \MR{2017446}

\bibitem[Lur09]{MR2522659}
Jacob Lurie, \emph{Higher topos theory}, Annals of Mathematics Studies, vol.
  170, Princeton University Press, Princeton, NJ, 2009. \MR{2522659}

\bibitem[Lur17]{lurie2017higher}
\bysame, \emph{Higher algebra},
  \url{https://people.math.harvard.edu/~lurie/papers/HA.pdf}, 2017, preprint.

\bibitem[Lur18]{lurie2018spectral}
\bysame, \emph{Spectral algebraic geometry},
  \url{https://www.math.ias.edu/~lurie/papers/SAG-rootfile.pdf}, February 2018,
  preprint.

\bibitem[Sal01]{MR1851264}
Paolo Salvatore, \emph{Configuration spaces with summable labels},
  Cohomological methods in homotopy theory ({B}ellaterra, 1998), Progr. Math.,
  vol. 196, Birkh\"{a}user, Basel, 2001, pp.~375--395. \MR{1851264}

\bibitem[Seg10]{MR2681691}
Graeme Segal, \emph{Locality of holomorphic bundles, and locality in quantum
  field theory}, The many facets of geometry, Oxford Univ. Press, Oxford, 2010,
  pp.~164--176. \MR{2681691}

\bibitem[SW01]{salvatore2001frameddiscsoperadsequivariant}
Paolo Salvatore and Nathalie Wahl, \emph{Framed discs operads and the
  equivariant recognition principle}, \url{https://arxiv.org/abs/math/0106242},
  2001, arXiv preprint.

\bibitem[Zha22]{zhao2022metaplectic}
Yifei Zhao, \emph{\'{E}tale metaplectic covers of reductive group schemes},
  \url{https://arxiv.org/abs/2204.00610}, 2022.

\end{thebibliography}

\end{document}